\documentclass[USenglish,onecolumn]{article}

\usepackage[left=3cm,right=3cm,top=3cm,bottom=3cm]{geometry} 
\usepackage[utf8]{inputenc} 
\usepackage{amsmath}        
\usepackage{mathtools}      
\usepackage{amsfonts}       
\usepackage{amssymb}        
\usepackage{easybmat}       
\usepackage{arydshln}       
\usepackage{enumitem}       
\usepackage{algorithm}      
\usepackage{algorithmicx}   
\usepackage[noend]{algpseudocode}  
\usepackage{hyperref}       
\hypersetup{
colorlinks,
citecolor=[rgb]{.0,.5,.0},
linkcolor=[rgb]{.0,.5,.0},
urlcolor=black,
hypertexnames=true
}
\usepackage{cleveref}       
\usepackage{setspace}       
\usepackage{tocloft}       
\usepackage{amsthm}
\newtheorem{theorem}{Theorem}
\theoremstyle{plain}
\newtheorem{proposition}[theorem]{Proposition}
\theoremstyle{plain}
\newtheorem{corollary}[theorem]{Corollary}
\theoremstyle{remark}
\newtheorem*{remark*}{Remark}
\theoremstyle{remark}
\newtheorem{remark}{Remark}
\theoremstyle{plain}
\newtheorem{lemma}[theorem]{Lemma}
\theoremstyle{definition}
\newtheorem{example}[theorem]{Example}
\theoremstyle{definition}
\newtheorem{definition}[theorem]{Definition}
\theoremstyle{plain}
\newtheorem*{conjecture*}{Conjecture}

\numberwithin{theorem}{section}
\numberwithin{equation}{section}
\numberwithin{figure}{section}

\usepackage{tensor}    
\newcommand{\vfromk}[2]{\tensor*[^{#2}]{\vec{\mathbf{#1}}}{}}
\newcommand{\vintok}[2]{\tensor*{\vec{\mathbf{#1}}}{^{{#2}}}}

\DeclareMathOperator{\adj}{adj}
\DeclareMathOperator{\coll}{coll}
\DeclareMathOperator{\Tr}{Tr}

\DeclareMathOperator{\am}{AM}
\DeclareMathOperator{\gm}{GM}

\usepackage{stmaryrd}
\usepackage{trimclip}

\makeatletter
\DeclareRobustCommand{\shortto}{%
  \mathrel{\mathpalette\short@to\relax}%
}

\newcommand{\short@to}[2]{%
  \mkern2mu
  \clipbox{{.5\width} 0 0 0}{$\m@th#1\vphantom{+}{\shortrightarrow}$}%
  }
\makeatother

\usepackage{xcolor}

\newcommand{\rr}[1]{\texorpdfstring{{\color{blue}#1}}{}}

\usepackage{nicematrix}    
\NiceMatrixOptions{
code-for-first-col = $\bfseries\mathversion{bold}$,
code-for-first-row = $\bfseries\mathversion{bold}$\rotate
}
\usepackage{tikz}         
\usetikzlibrary{fit}      
\begin{document}

\title{Geometric multiplicity of unitary non-backtracking eigenvalues}

\author{
Leo Torres \\
Nora Institute \\
\url{leo@leotrs.com}
}
\date{}

\maketitle

\begin{abstract}{
We characterize the conditions under which a complex unitary number is an eigenvalue of the non-backtracking matrix of an undirected graph.
Further, we provide a closed formula to compute its geometric multiplicity and an algorithm to compute this multiplicity without making a single matrix computation.
The algorithm has time complexity that is linear in the size of the graph.


}
\end{abstract}

\begingroup
\hypersetup{hidelinks}
\setcounter{tocdepth}{2}
\tableofcontents
\endgroup


\newpage
\section{Introduction}\label{sec:intro}

\rr{Given an undirected graph $G$, its \emph{non-backtracking matrix} is a binary matrix that encodes the incidence of oriented edges of $G$;
we give a formal definition in the next Section.}
It was first introduced by Hashimoto in the study of Zeta functions of graphs \cite{hashimoto1989zeta}.
He and others proved that the poles of the Zeta function of a graph coincide with the reciprocals of the eigenvalues of the non-backtracking (NB) matrix \cite{bass1992ihara,terras2010zeta}.
Subsequently, the NB matrix became a point of focus of modern spectral graph theory.
It is now known that the eigenvalues of the NB matrix (or NB eigenvalues for short), independently of their relationship to the Zeta function, contain rich information about the structure of the underlying graph, and have been used in multiple applications and theoretical settings \cite{bordenave2018nonbacktracking,cooper2009properties,torres2021nonbacktracking,Torres2019}.

\rr{Since the NB matrix is binary}, Perron-Frobenius theory can describe the eigenvalues $\lambda$ that have maximal magnitude.
At the other end of the spectrum, the eigenvalue $\lambda = 0$ can be shown to be related to the existence of nodes of degree $1$, or more generally, tree-like subgraphs.
Up to now, these were the only two families of NB eigenvalues whose properties have been fully described for any arbitrary graph.
The difficulty in making general claims about the NB eigenvalues is that, in contrast to the classical matrices studied in spectral graph theory, the NB matrix is not symmetric, in fact it is not even normal.
Thus most of the existing tools from spectral graph theory do not apply to it.

Here, we study a third family of eigenvalues, namely those complex numbers on the unit circle, and fully describe how they are related to the structure of the graph.
First, we show that if a complex unitary number is a NB eigenvalue then it must be a root of unity.
Then, we characterize the subgraphs that act as the support of eigenfunctions corresponding to unitary eigenvalues, and describe all possible ways in which these subgraphs may be connected to the rest of the graph.
Finally, we derive closed formulas that count the geometric multiplicity of unitary NB eigenvalues, and we design an algorithm to compute these formulas efficiently.

\rr{
Beyond providing a more complete picture of how the NB eigenvalues are related to the structure of a graph, the particular topic of multiplicity is of independent interest.
Indeed, it is known that certain ensembles of random matrices have simple spectrum \cite{tao2017random}, and these include the adjacency matrices of certain random graphs.
However, our results on the multiplicity of unitary NB eigenvalues explicitly prevent any such results to be extended to the NB matrices of random graphs in full generality.
At the very least, they point to the fact that the unitary NB eigenvalues could be treated as a special case in such situations. 
Furthermore, our algorithm to compute the geometric multiplicity of all unitary eigenvalues is more efficient than standard eigensolvers and can thus be used alongside other standard tools to compute the full NB spectrum of a graph more efficiently.}

The rest of this manuscript is structured as follows.
In \Cref{sec:background} we cover all necessary preliminaries.
In \Cref{sec:real} we discuss the real case, that is, we describe the NB eigenvalues $\lambda = \pm 1$ and compute their geometric and algebraic multiplicities.
Subsequent sections are dedicated solely to the nonreal case.
\Cref{sec:eigenfunctions} establishes the first properties of nonreal unitary eigenvalues and their eigenfunctions, and we show that a subgraph that is the support of an eigenfunction must always be a \emph{graph subdivision}.
The $r$-subdivision of a graph $G$ is built by replacing each edge in $G$ by a chain of $r$ edges, joined by nodes of degree $2$.
\Cref{sec:subdivisions} studies the special case of graph subdivisions.
In \Cref{sec:gluing} we determine all possible ways in which a graph subdivision may be \emph{glued} to some other graph while leaving the unitary NB eigenfunctions intact, and we argue this procedure is the only mechanism by which unitary numbers appear in the NB spectrum of any graph.
Finally, \Cref{sec:geometric} states and proves our main results.

\section{Background}\label{sec:background}
\paragraph*{Conventions and notation}
A graph $G$ is comprised of a set of nodes $V(G)$ and a set of edges $E(G)$.
All graphs considered are finite, undirected, unweighted, connected, and possibly contain self-loops or multi-edges.
The degree of node $v \in V(G)$ is the number of edges that contain it and is denoted $d_v$.
A self-loop $(v, v) \in E(G)$ adds two units to the degree.

If a complex number $z$ satisfies $|z| = 1$, we say it is \emph{unitary} (a.k.a. \emph{unimodular}).
Given a unitary $\lambda$, if there exists an integer $q$ such that $\lambda^q = 1$ then $\lambda$ is said to be a $q$th root of unity.
The smallest positive such $q$ is the \emph{order of primitivity} of $\lambda$, and in this case $\lambda$ is a \emph{primitive $q$th root} of unity.

\paragraph*{Multiplicity of eigenvalues}
Let $\mathbf{A}$ be a square matrix and let $(\lambda, \mathbf{v})$ be a pair such that $\mathbf{A v} = \lambda \mathbf{v}$.
The \emph{algebraic multiplicity} of $\lambda$, denoted $\am_{\mathbf{A}}(\lambda)$, is the multiplicity of $\lambda$ as a root of the polynomial $\det \left( \mathbf{A} - t \mathbf{I} \right)$.
The \emph{geometric multiplicity} of $\lambda$, denoted $\gm_{\mathbf{A}}(\lambda)$, is the dimension of the nullspace of $\left( \mathbf{A} - \lambda \mathbf{I} \right)$.
It is always the case that $\am_{\mathbf{A}}(\lambda) \geq \gm_{\mathbf{A}}(\lambda)$.
If $G$ is a graph, we write $\am_{G}(\lambda), \gm_{G}(\lambda)$ to refer to the multiplicities of $\lambda$ as an eigenvalue of the non-backtracking matrix of $G$.
We drop the subscript when clear from context.

\paragraph*{NB walks and NB matrix}
Consider a graph $G$ with $n$ nodes and $m$ edges.
Let $\bar{E} $ be the set of \emph{oriented} edges, i.e., if nodes $u$ and $v$ are neighbors in $G$, both oriented edges $u \to v$ and $v \to u$ are members of $\bar{E}$.
A \emph{walk} in $G$ is a sequence of oriented edges that are consecutively adjacent.
The \emph{length} of a walk is the number of oriented edges in it.
A walk of the type $u \to v \to u$ is called a \emph{backtrack}.
A \emph{non-backtracking} walk, or NB walk, is a walk that does not contain any backtracks.
The \emph{NB matrix} of $G$ is denoted by $\mathbf{B}$ and can be understood as the adjacency matrix of its oriented edges, ignoring backtracks.
Concretely, $\mathbf{B}$ is indexed in the rows and columns by $\bar{E}$, \rr{where we fix an arbitrary ordering of the edges}, and is defined as
\begin{equation}\label{eqn:def-b}
\mathbf{B}_{k \to l, i \to j} = \mathbf{1}\{j=k\} \left( 1 - \mathbf{1}\{i=l\} \right).
\end{equation}
In words, $\mathbf{B}_{k \to l, i \to j}$ is one whenever $k$ equals $j$ and the walk $i \to j = k \to l$ is a valid walk of length 2 that is not a backtrack, and it is zero otherwise.
Importantly, the quantity $\mathbf{B}^r_{k \to l, i \to j}$ is equal to the number of NB walks starting with $i \to j$ and ending with $k \to l$ of length $r + 1$.\footnote{
The notation $\mathbf{B}_{k \to l, i \to j}$ is, strictly speaking, unambiguous only when the graph simple, i.e., when it contains no self-loops or multi-edges.
When the graph is not simple, we should instead use the following notation: for an oriented edge $e \in \bar{E}$, write $s(e)$ and $t(e)$ for its source and target endpoints.
Then, the NB matrix is indexed as $\mathbf{B}_{e,e'}$ or as $\mathbf{B}_{s(e) \to t(e), s(e') \to t(e')}$.
We use the former notation, so as to not complicate our formulas with multiple repetitions of $s(\cdot)$ and $t(\cdot)$.
Notational choices notwithstanding, we reiterate that every result in this manuscript applies regardless of whether or not the graph is simple.
}

Let $\mathbf{v}$ be a function defined on $\bar{E}$ and let $\mathbf{A} = \left( a_{ij} \right)$ be the adjacency matrix of $G$.
In this case, \cref{eqn:def-b} yields (see \cref{fig:nbm-doodle})
\begin{equation}\label{eqn:operator}
\left( \mathbf{B v} \right)_{k \to l} = \sum_{i} a_{ij} \mathbf{v}_{i \to k} - \mathbf{v}_{l \to k}.
\end{equation}

\begin{figure}[t]
\centering
\includegraphics[width=0.82\textwidth]{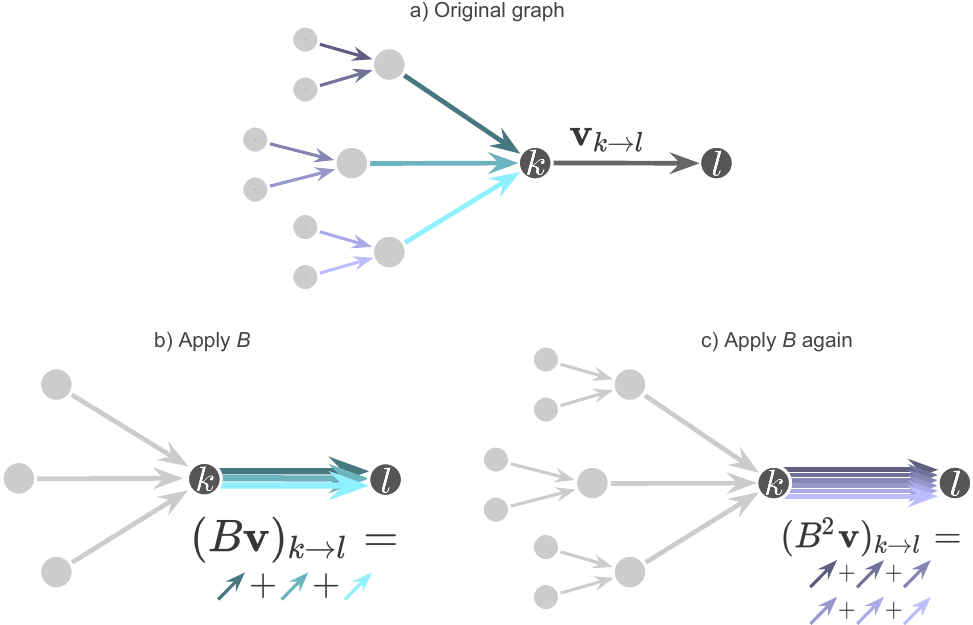}
\caption{\label{fig:nbm-doodle}
\textbf{Action of non-backtracking matrix.}
\textbf{a)} Edge color represents the value assigned to that directed edge by the function $\mathbf{v}$.
\textbf{b)} $\mathbf{B v}$ aggregates the values along all incoming edges, except for the backtrack $\mathbf{v}_{l \to k}$.
Overlapping arrows represent the sum of the corresponding color-coded quantities.
\textbf{c)} $\mathbf{B}^{2} \mathbf{v}$ aggregates the values along all NB walks of length $3$.}
\end{figure}
The first term in the right-hand side will appear many times in our discussion so we define, for any function $\mathbf{v}$ and node $k$,
\begin{equation}\label{eqn:def-vintok}
\vintok{v}{k} \coloneqq \sum_{i} a_{ik} \mathbf{v}_{i \to k}
\quad
\text{and}
\quad
\vfromk{v}{k} \coloneqq \sum_{i} a_{ik} \mathbf{v}_{k \to i},
\end{equation}
which we pronounce ``\emph{$\mathbf{v}$ into k}'' and ``\emph{$\mathbf{v}$ from k}'', respectively.

\paragraph*{NB eigenvalues and eigenfunctions} 
Let $\mathbf{D}$ be the diagonal degree matrix.
Then, the \emph{Ihara determinant formula} (\cite{angel2007non,kotani}) states
\begin{equation}\label{eqn:ihara}
\det \left( I - t \mathbf{B} \right) = \left( 1 - t^2 \right)^{m - n} \det \left( I - t \mathbf{A} + t^2 \left( \mathbf{D} - \mathbf{I} \right) \right).
\end{equation}
Note the left-hand side of \cref{eqn:ihara} is the characteristic polynomial of $\mathbf{B}$ evaluated at $1/t$.
Importantly, the Ihara determinant formula works on any (simple or non-simple) graph even if it is disconnected, though we continue to assume connectedness for the sake of simplicity.

Now suppose $\mathbf{v}$ is an eigenfunction of $\mathbf{B}$ with eigenvalue $\lambda$.
Then \cref{eqn:operator} yields
\begin{equation}\label{eqn:evector}
\lambda \mathbf{v}_{k \to l} + \mathbf{v}_{l \to k} = \vintok{v}{k}.
\end{equation}
Since $\mathbf{B}$ is a non-negative matrix, Perron-Frobenius theory implies that when it is also irreducible, it admits a simple, real and positive eigenvalue whose corresponding eigenfunction has all positive entries.
This is called the Perron eigenfunction.
If $\mathbf{v}$ is the Perron eigenfunction of $\mathbf{B}$, the quantity $\vintok{v}{k}$ is called \textit{non-backtracking centrality} of node $k$ by some authors \cite{martin2014localization}.

It is known (see \cite{Grindrod2018,torres2020non}) that the nodes of degree $1$ do not affect the nonzero part of the spectrum, thus in the rest of the manuscript we assume that each node in $G$ has degree greater than $1$.

\begin{remark}[\textbf{Circle graphs}]\label{rmk:circle}
The special case when $G$ is a circle graph is important for our later discussion.
Suppose $G$ is a circle graph of length $r$.
In this case $\mathbf{B}$ is the direct sum of two cyclic permutations going around the cycle in opposite directions.
Hence, the NB eigenvalues are precisely the $r$th roots of unity, each with multiplicity two, and there exists a full basis of eigenfunctions.
\end{remark}

\section{The real case}\label{sec:real}
We begin our discussion with the real roots of unity, namely $\pm1$.
The algebraic multiplicity of $\pm1$ as an NB eigenvalue is well-known \cite{Grindrod2018,hashimoto1989zeta,northshield1998}, though we include the full derivations here for the sake of completeness. 
The geometric multiplicities are a novel result, as far as the author can tell, though \cite{horton2006ihara} studies the eigenfunctions of $+1$.

Consider a graph $G$ with $n$ nodes and $m$ edges.
Recall we assume that $G$ is connected and contains no nodes of degree $1$.
In particular, $G$ is not a forest, and we have $m \geq n$, with equality if and only if $G$ is a circle graph.
Let $\mathbf{B}$ be the NB matrix of $G$.
\Cref{eqn:ihara} implies that $\pm 1$ are always eigenvalues of $\mathbf{B}$, each with algebraic multiplicity at least $m - n$.
Here we establish that $\am(1) = \gm(1) = m - n + 1$, and $\am(-1) = \gm(-1) = m - n + 1$ when $G$ is bipartite and $\am(-1) = \gm(-1) = m - n$ otherwise.

\subsection{Algebraic multiplicity}\label{sec:real-alg}

Recall that $\mathbf{D-A}$ is called the \emph{Laplacian }matrix of $G$, which is always singular, and whose rank is $n-1$ if and only if the graph is connected.
An argument closely related to the following \rr{two proofs} can be found in \cite{Grindrod2018,northshield1998}. 

\begin{proposition}
\label{pro:am-plus1}Let $G$ have at least two cycles. Then $\am(1)=m-n+1$. 
\end{proposition}

\begin{proof}
Define $f(t) \coloneqq \det \left(\mathbf{I} - t \mathbf{A} + t^{2}(\mathbf{D-I})\right)$ and observe that $f(1) = \det \left( \mathbf{D - A} \right) = 0$.
Therefore, $f(t) = (t-1) g(t)$ for some $g(t)$.
This fact, together with \cref{eqn:ihara} imply $\am(1) \geq m - n + 1$.
Thus, showing $g(1) \neq 0$ finishes the proof.
First, note that $g(1)$ equals $f'(1)$.
The so-called Jacobi formula shows $f'(1) = \Tr \left[ \adj \left( \mathbf{D - A} \right) \left( \mathbf{D - A + D} - 2\mathbf{I} \right) \right]$; see e.g. eq. (41) of \cite{petersen2012matrix}.
Further, well-known properties of the adjugate show that $\adj \left( \mathbf{D - A} \right) = \eta \mathbf{1} \mathbf{1}^{\top}$ for some nonzero $\eta$; see e.g. \cite{merris1994laplacian} and section 0.8.2 of \cite{horn2012matrix}.
Finally, observe that $m > n$ since $G$ is not a circle graph.
All together, we have 
\begin{align}
g(1) = f'(1)
 & = \Tr \left[ \adj \left( \mathbf{D - A} \right) \left( \mathbf{D - A + D} - 2 \mathbf{I}\right) \right]  \\
 & = \eta \Tr \left[ \mathbf{1} \mathbf{1}^{\top} \left( \mathbf{D - A} \right) + \mathbf{11}^{\top} \left( \mathbf{D} - 2 \mathbf{I} \right) \right] \nonumber \\
 & = \eta \mathbf{1}^{\top} \left( \mathbf{D - A} \right) \mathbf{1} + \eta \mathbf{1}^{\top} \left( \mathbf{D} - 2 \mathbf{I} \right) \mathbf{1} \nonumber \\
 & = \eta \left(2m-2n\right) \neq 0 \nonumber \qedhere
\end{align}
\end{proof}

For our treatment of the eigenvalue $-1$, recall that $\mathbf{D+A}$ is called the \emph{signless Laplacian} of $G$.
\begin{proposition}\label{pro:am-minus1}
Let $G$ have at least two cycles.
Then $\am(-1) = m - n + \mathbf{1}\{G \text{ is bipartite}\}$. 
\end{proposition}

\begin{proof}
Define $f(t)$ as in Proposition \ref{pro:am-plus1} and observe that $f(-1) = \det \left( \mathbf{D + A} \right)$.
It is known that $\mathbf{D + A}$ is singular if and only if $G$ is bipartite \cite{cvetkovic2007signless}.
Therefore, if $G$ is not bipartite, then $f(-1) \neq 0$ and $\am(-1)=m-n$ due to \cref{eqn:ihara}.
If $G$ is bipartite, we have $f(-1) = 0$ and thus we can write $f(t) = (1 + t) g(t)$.
To finish, we show $g(-1) \neq 0$.
Let the partition of the node set be $U_{1}$ and $U_{2}$ and define the node function $\mathbf{u}$ by putting $\mathbf{u}_{i}=1$ if $i \in U_{1}$ and $\mathbf{u}_{j} = -1$ if $j\in U_{2}$.
One can check that $\left( \mathbf{D + A}\right) \mathbf{u} = 0$, which implies $\adj \left( \mathbf{D + A} \right) = \eta \mathbf{u} \mathbf{u}^{\top}$ for some scalar $\eta$.
Finally, a similar procedure as in Proposition \ref{pro:am-plus1} shows that 
\begin{equation}
g(-1) = f'(-1) = \Tr \left[ \adj \left( \mathbf{D + A} \right) \left( 2\mathbf{I - D} \right) \right] = \eta \mathbf{u}^{\top} \left( 2 \mathbf{I - D} \right) \mathbf{u} = \eta \left( 2n - 2m \right) \neq 0. \qedhere
\end{equation}
\end{proof}

\subsection{Geometric multiplicity}\label{sec:real-geo}

\begin{proposition}
\label{pro:gm-plus1}Let $G$ have at least two cycles.
Then $\gm(1) = m - n + 1$. 
\end{proposition}

\begin{proof}
Since $\gm$ is bounded above by $\am$, we have $\gm(1) \leq m - n + 1$.
Thus we only need to show that there exists a set of $m-n+1$ linearly independent functions that satisfy $\mathbf{B v}=\mathbf{v}$.
Inspection of \cref{eqn:evector} when $\lambda=1$ shows that there exists a global constant $h$ such that 
\begin{equation}
\mathbf{v}_{k \to l} + \mathbf{v}_{l \to k} = h =\vintok{v}{k},
\end{equation}
for any neighboring nodes $k,l$.
Summing the left equation over each of the $m$ edges yields 
\begin{equation}
\frac{1}{2}\sum_{k,l}a_{kl}\left(\mathbf{v}_{k\to l}+\mathbf{v}_{l\to k}\right)=mh,
\end{equation}
while summing the right equation over each of the $n$ nodes yields
\begin{equation}
nh = \sum_{k} \vintok{v}{k}.
\end{equation}
Note the sums in the last two equations sum each of the components of $\mathbf{v}$ exactly once and thus $nh=mh$.
Since $G$ is not a circle graph, we have $m>n$ and $h=0$.
In other words, $\mathbf{v}$ satisfies the system 
\begin{equation}\label{eqn:system-1}
\begin{cases}
\vintok{v}{k} = 0 & \text{for each node } k, \\
\mathbf{v}_{k \to l} + \mathbf{v}_{l \to k} = 0 & \text{for each edge } k-l.
\end{cases}
\end{equation}

We finish by arguing there are exactly $m-n+1$ linearly independent solutions of \cref{eqn:system-1}.
Suppose the nodes $\mathcal{C} = \{ v_1, \ldots, v_r \}$ form a non-self-intersecting cycle in $G$, and consider the vector $\mathbf{v}^\mathcal{C}$ with components
\begin{equation}
\mathbf{v}^\mathcal{C}_{v_i \to v_{i+1}} \coloneqq 1 \text{ and } \mathbf{v}^\mathcal{C}_{v_{i+1} \to v_{i}} \coloneqq -1, 
\end{equation}
where the index $i$ is taken modulo $r$, and zero elsewhere.
Then, $\mathbf{v}^\mathcal{C}$ satisfies \cref{eqn:system-1}.
Now take a spanning tree $T$ of $G$.
Each edge not in $T$ defines a non-self-intersecting cycle in $G$, and there are $m-n+1$ such edges, say $\{e_1, \ldots, e_{m-n+1}\}$.
Let $\{\mathcal{C}_1, \ldots, \mathcal{C}_{m-n+1}\}$ be the corresponding set of cycles, each $\mathcal{C}_j$ being defined by $e_j$.
For each cycle $\mathcal{C}_j$, construct the vector $\mathbf{v}^{\mathcal{C}_j}$ as before, that is, with components equal to $+1$ on edges going around $\mathcal{C}_j$ in one direction, components equal to $-1$ going around $\mathcal{C}_j$ in the other direction, and zero elsewhere.
Each vector $\mathbf{v}^{\mathcal{C}_j}$ is a solution of \cref{eqn:system-1}.
Finally, $\{ \mathbf{v}^{\mathcal{C}_j} \}$ is a linearly independent set because each $\mathbf{v}^{\mathcal{C}_j}$ is nonzero in the component corresponding to $e_j$, while for each $i \neq j$, $\mathbf{v}_i$ is $0$ on that component.
\end{proof}

\begin{corollary}
$\gm(1)$ is the number of linearly independent ways there are to assign electrical current flows to the edges of $G$ in such a way that they satisfy Kirchoff's law of circuits. 
\end{corollary}

\begin{proof}
If $G$ represents an electrical circuit and the value of $\mathbf{v}_{k\to l}$ is the current flow in the direction of $k$ to $l$, then \cref{eqn:system-1} is exactly equivalent to Kirchoff's law. 
\end{proof}

\begin{proposition}\label{pro:gm-minus1}
Let $G$ have at least two cycles.
Then $\gm(-1) = m - n + 1$ if $G$ is bipartite and $\gm(-1) = m - n$ if $G$ is not bipartite. 
\end{proposition}

\begin{proof}
Inspection of \cref{eqn:evector} when $\lambda = -1$ and an argument similar to that in Proposition \ref{pro:gm-plus1} shows that if $\mathbf{B v} = -\mathbf{v}$ then 
\begin{equation}
\begin{cases}
\vintok{v}{k} = 0 & \text{for each node } k,\\
\mathbf{v}_{k \to l} = \mathbf{v}_{l \to k} & \text{for each edge } k-l.
\end{cases}\label{eqn:system-minus-1}
\end{equation}
To fix ideas, suppose the nodes $\mathcal{C} = \{w, x, y, z\}$ form a cycle of length $4$ in $G$, and consider the function $\mathbf{v}$ with $\mathbf{v}_{w \to x} = \mathbf{v}_{x \to w} = \mathbf{v}_{y \to z} = \mathbf{v}_{z \to y} = 1$, and $\mathbf{v}_{x \to y} = \mathbf{v}_{y \to x} = \mathbf{v}_{z \to w} = \mathbf{v}_{w \to z} = -1$, and zero elsewhere.
Then, $\mathbf{v}$ satisfies \cref{eqn:system-minus-1}.
In general, a vector thusly constructed from a cycle $\mathcal{C}$ of even length will satisfy \cref{eqn:system-minus-1}.
The dimension of the space spanned by the even-length cycles has been studied in \cite{rowlinson2004,simic2015,godsil2013algebraic}, and it is known to be $m-n+1$ when $G$ is bipartite and $m-n$ otherwise. 
\end{proof}

\begin{corollary}\label{cor:bases}
The eigenspace of $\lambda = 1$ admits a basis where each element is supported on a different cycle, while the eigenspace of $\lambda = -1$ admits a basis where each element is supported on a different cycle of even length.
\end{corollary}

\begin{proof}
These bases were exhibited in the proofs of Propositions \ref{pro:gm-plus1} and \ref{pro:gm-minus1}.
\end{proof}



The rest of this manuscript is dedicated to study the nonreal unitary NB eigenvalues.
This section showed that the real eigenvalues are related to the existence of arbitrary cycles in $G$.
The next sections will establish that the existence of nonreal eigenvalues is related to the existence of a particular class of subgraphs, namely graphs that are \emph{subdivisions} of other graphs.

\section{Eigenfunctions of unitary eigenvalue}\label{sec:eigenfunctions}

\rr{
Suppose $\mathbf{B} \mathbf{v} = \lambda \mathbf{v}$.  The following equivalence will be key to our discussion.

\begin{equation}\label{eqn:bbt}
|\lambda| = 1 \quad \iff \quad \mathbf{B^{*} B} \mathbf{v} = \mathbf{B B^{*}} \mathbf{v} = \mathbf{v}.
\end{equation}

The forward direction is proved in Appendix \ref{app:indefinite} using techniques from indefinite linear algebra.  The reverse direction is immediate: if $\mathbf{B^{*}} \mathbf{B} \mathbf{v} = \mathbf{v}$ and $\mathbf{B} \mathbf{v} = \lambda \mathbf{v}$, then $\langle \mathbf{v}, \mathbf{v} \rangle = \langle \mathbf{v}, \mathbf{B^{*}} \mathbf{B} \mathbf{v} \rangle = \langle \mathbf{B} \mathbf{v}, \mathbf{B} \mathbf{v} \rangle = |\lambda|^2 \langle \mathbf{v}, \mathbf{v} \rangle$, thus $|\lambda| = 1$.
} In view of this, we start by computing $\mathbf{B^{*} B}$ and $\mathbf{B B^{*}}$; see \cref{fig:bbt}.

\subsection{Non-leaky functions}\label{sec:detachment}

\begin{figure}[t]
\begin{centering}
\includegraphics[width=0.9\textwidth]{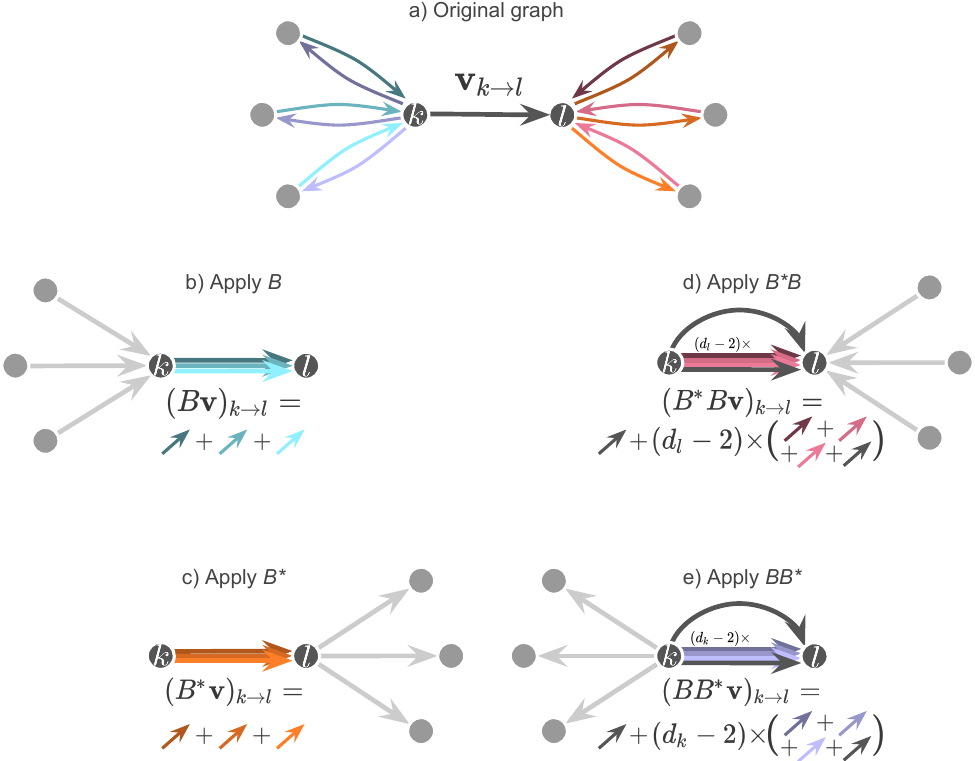}
\par\end{centering}
\caption{\label{fig:bbt}
\textbf{The action of $\mathbf{B,B^{*},B^{*}B}$ and $\mathbf{BB^{*}}$ on a function $\mathbf{v}$}.
\textbf{a)} The value of $\mathbf{v}$ at each directed edge is color-coded.
\textbf{b)} The value of $\left( \mathbf{B v}\right)_{k \to l}$ equals the sum of the values along edges incoming to $k$, ignoring the backtrack; cf. \cref{fig:nbm-doodle}.
Overlapping arrows represent the sum of the corresponding color-coded quantities.
\textbf{c-e)} The result of applying $\mathbf{B^*, B^* B, B B^*}$ is similarly depicted. 
}
\end{figure}

\begin{lemma}\label{lem:bbt}
Let $\mathbf{B}$ be the NB matrix of any graph $G$.
Then for any function $\mathbf{v}$, the following hold
\begin{equation}
\left( \mathbf{B^{*} B} \mathbf{v} \right)_{k \to l} = \left( d_{l} - 2 \right) \vintok{v}{l} + \mathbf{v}_{k \to l} \quad\quad \text{and} \quad\quad \left( \mathbf{B B^{*}} \mathbf{v} \right)_{k \to l} = \left(d_{k} - 2 \right) \vfromk{v}{k} + \mathbf{v}_{k \to l}.
\end{equation}
\end{lemma}

\begin{proof}
This is direct from the definition of $\mathbf{B}$.
For brevity, we show only the case of $\mathbf{B^{*}B v}$. 
\begin{align}
\left( \mathbf{B^{*}B} \mathbf{v} \right)_{k \to l}
 &= \sum_{i \to j} \mathbf{1}\{i = l\} \left( 1 - \mathbf{1}\{j=k\} \right) \sum_{r \to s} \mathbf{1}
 \{i = s\} \left( 1 - \mathbf{1}\{j=r\} \right) \mathbf{v}_{r \to s} \\
 &= \sum_{ j \neq k }a_{jl} \left( \sum_{r} a_{rl} \mathbf{v}_{r \to l} - \mathbf{v}_{j \to l} \right) \nonumber \\
 & = \left( d_{l} - 1 \right) \vintok{v}{l} - \vintok{v}{l} + \mathbf{v}_{k \to l}. \nonumber \qedhere
\end{align}
\end{proof}

Lemma \ref{lem:bbt} features the first appearance of $\vfromk{v}{k}$ (``\emph    {$\mathbf{v}$ from $k$}'') in our discussion.
When $\mathbf{v}$ is an eigenfunction, $\vintok{v}{k}$ and $\vfromk{v}{k}$ are related as follows.

\begin{lemma}[See \cite{lin2019non}]\label{lem:NB-in-cent}
Suppose $\mathbf{B v} = \lambda \mathbf{v}$.
Then for any node $k$ we have 
\begin{equation}
\left( d_{k} - 1 \right) \vintok{v}{k} = \lambda\,\vfromk{v}{k}.
\end{equation}
\end{lemma}

\begin{proof}
Replacing \cref{eqn:def-vintok} in \cref{eqn:evector} and summing over all neighbors of $k$ we obtain the result.
\end{proof}

Lemma \ref{lem:NB-in-cent} implies that for an eigenfunction $\mathbf{v}$, we have $\vintok{v}{k} = 0$ if and only if $\vfromk{v}{k} = 0$ (recall we only consider graphs with minimum degree at least $2$).
Therefore, and in view of \cref{eqn:bbt} and Lemma \ref{lem:bbt}, to understand the eigenfunctions of unitary eigenvalues, it is sufficient to understand those functions $\mathbf{v}$ that satisfy $\left( d_{k} - 2 \right) \vintok{v}{k} = 0$ for each node $k$.
The following example and the accompanying \cref{fig:cycle-support} provide a structural interpretation of this condition.

\begin{example}[\bf{Eigenfunction supported on a cycle}]\label{exm:cycle-support}
Let $G$ be a graph with NB matrix $\mathbf{B}$ and let $\mathcal{D}$ be a set of nodes whose induced subgraph is a cycle.
Let $k \in \mathcal{D}$ and $l \notin \mathcal{D}$ be neighboring nodes.
Suppose $\mathbf{B v}=\lambda \mathbf{v}$ with $\lambda \neq 0$ and assume $\mathbf{v}$ is supported solely on the edges among the nodes in $\mathcal{D}$.
Since $k \to l$ and $l \to k$ are not in the support of $\mathbf{v}$, we have
\begin{equation}
0 = \lambda \mathbf{v}_{k \to l} = \left( \mathbf{B v} \right)_{k \to l} = \vintok{v}{k} - \mathbf{v}_{l \to k} = \vintok{v}{k}.
\end{equation}
Thus, a necessary condition for $\mathbf{v}$ to be an eigenfunction supported on a cycle $\mathcal{D}$ is that for every $k$ in $\mathcal{D}$ with a neighbor not in $\mathcal{D}$, we must have $\vintok{v}{k} = 0$.
If $k \in \mathcal{D}$ has no neighbors outside of $\mathcal{D}$, i.e. if its degree is $2$, then there is no restriction on $\vintok{v}{k}$.
In other words, $\mathbf{v}$ satisfies $\left( d_{k} - 2 \right) \vintok{v}{k} = 0$ for each $k$, and therefore $\mathbf{B^{*} B v} = \mathbf{v}$ by Lemma \ref{lem:bbt}, and $\mathbf{B B^{*} v} = \mathbf{v}$ due to Lemma \ref{lem:NB-in-cent}.
Therefore, $\lambda$ is unitary. Lastly, if $r$ is the length of the cycle induced by $\mathcal{D}$, we have $\mathbf{B}^{r} \mathbf{v} = \mathbf{v}$ and thus $\lambda^{r}=1$.
\end{example}

\begin{figure}
\centering
\includegraphics[width=\textwidth]{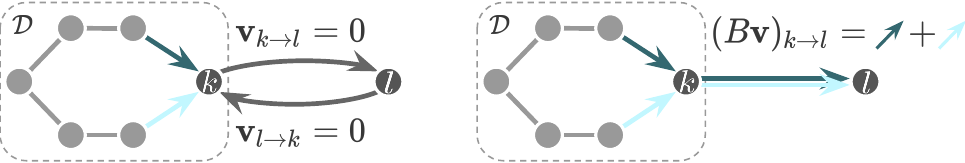}
\caption{\label{fig:cycle-support}
\textbf{Example of leakiness.}
The nodes in $\mathcal{D}$ induce a cycle. The node $k\in\mathcal{D}$ has a neighbor $l\protect\notin\mathcal{D}$.
\textbf{Left:} a function $\mathbf{v}$ supported on $\mathcal{D}$.
\textbf{Right:} if $\mathbf{v}$ is an eigenfunction, the sum of all values incoming to $k$ must be zero.
If $\left( \mathbf{Bv} \right)_{k \to l}$ is non-zero, we say that $\mathbf{v}$ \emph{leaks out} of $\mathcal{D}$ via $k$.
Nodes with degree $2$ can never leak.}
\end{figure}

The following terminology captures the property $\left( d_{k} - 2 \right) \vintok{v}{k} = 0$, inspired by Example \ref{exm:cycle-support}. 

\begin{definition}[\bf{Non-leaky functions}]\label{def:leaky}
Consider a function $\mathbf{v}$ (not necessarily an eigenfunction) with support $\mathcal{D}$ (not necessarily a cycle).
If there is a $k \in \mathcal{D}$ such that $\left( d_{k} - 2 \right) \vintok{v}{k} \neq 0$, we say that \emph{$\mathbf{v}$ leaks out of $\mathcal{D}$ via $k$}, or simply that \emph{$\mathbf{v}$ is leaky}.
If $\mathbf{v}$ does not leak via any node, we say that $\mathbf{v}$ is \emph{non-leaky}; see \cref{fig:cycle-support}.
Every NB eigenfunction $\mathbf{v}$ of unitary eigenvalue $\lambda$ is non-leaky, because it must satisfy $\mathbf{B B^{*} v} = \mathbf{B^{*} B v} = \mathbf{v}$; see Lemma \ref{lem:bbt}.
However, there are non-leaky functions that are not eigenfunctions, namely, linear combinations of eigenfunctions of different unitary eigenvalues.
\end{definition}

The utility of non-leaky functions, which include eigenfunctions of unitary eigenvalue, is that they can be studied by focusing only on their support and neglecting the rest of the graph, or, in other words, by \emph{detaching} their support from the rest of the graph.
In what follows, we will mainly use this idea implicitly by focusing only on the support of an eigenfunction and ignoring the rest of the graph.
In \cref{sec:gluing}, we define a construction that is converse to this detachment which we call \emph{gluing}.

\subsection{Support of an eigenfunction}\label{sec:support}
We now characterize those unitary complex numbers that may possibly be NB eigenvalues, and then we characterize the graphs that may possibly be the support of an eigenfunction of unitary eigenvalue.
Some of the following arguments require the graph to have no nodes of degree $1$.
However, as previously mentioned, the nodes with degree $1$ have no influence on the nonzero eigenvalues \cite{Grindrod2018,torres2020non}.
Thus, the results here are valid for arbitrary graphs, without restriction, as they concern only nonzero eigenvalues.

\begin{definition}
Given a NB eigenfunction $\mathbf{v}$ of a graph $G$, its \emph{support} is the set of edges $k - l$ such that at least one of $\mathbf{v}_{k \to l}$ or $\mathbf{v}_{l \to k}$ is nonzero. 
\end{definition}

\newpage

\begin{lemma}\label{lem:support-md2}
Let $G$ be a graph with no nodes of degree $1$ and NB matrix $\mathbf{B}$.
Suppose $\mathbf{B v} = \lambda \mathbf{v}$ with $\lambda \neq 0$ and let $H$ be the graph induced by the support of $\mathbf{v}$.
Then, $H$ has no nodes of degree $1$.
\end{lemma}

\begin{proof}
Assume there is a node $x \in H$ with degree $1$ in $H$ and degree $d>1$ in $G$.
Suppose $y$ is the only neighbor of $x$ that exists in $H$.
We show that both $\mathbf{v}_{x\to y}=0$ and $\mathbf{v}_{y\to x}=0$, which is a contradiction.

Since $\mathbf{v}$ is an eigenfunction, we have $\lambda\mathbf{v}_{x\to y} = \sum_{i \neq y} a_{ix} \mathbf{v}_{i \to x}$.
But the right-hand side is zero since otherwise $x$ would have degree greater than $1$ in $H$.
Since $\lambda \neq 0$, we have $\mathbf{v}_{x \to y} = 0$.

Let $z \in G \setminus H$ be a neighbor of $x$ with $z \neq y$.
We can always choose such $z$ because $d>1$.
We now have $0 = \lambda \mathbf{v}_{x \to z} = \mathbf{v}_{y \to x} + \sum_{i \neq y} a_{ix} \mathbf{v}_{i \to x} - \mathbf{v}_{z \to x} = \mathbf{v}_{y \to x}$, which finishes the proof.
\end{proof}

Most of the main results that follow are essentially corollaries of the upcoming Key Lemma, which determines the behavior of a NB eigenfunction on a specific type of subgraph called \emph{NB chain}.

\begin{definition}[\bf{NB chain}]
Suppose $G$ has a NB path of the form $k = i_{0}, i_{1}, \ldots, i_{p-1}, i_{p} = l$, where the nodes at the endpoints $k$ and $l$ have degrees strictly greater than $2$ and each $i_{j}$ has degree exactly $2$ for $j = 1, \ldots, p-1$.
This path is called an \emph{NB chain} of length $p$.
Note $k$ and $l$ may in fact be the same node, in which case this is a \emph{closed NB chain}.
If two nodes $k$ and $l$ with degrees greater than $2$ are neighbors, the path $k=i_{0},i_{1}=l$ is a NB chain of length one.
Note every graph with minimum degree at least $2$ that is not a cycle must always contain at least one NB chain.
\end{definition}

\begin{lemma}[\bf{Key Lemma}]\label{lem:path-degree-2}
Let $G$ have NB matrix $\mathbf{B}$ with $\mathbf{B v} = \lambda \mathbf{v}$ and $|\lambda|=1$.
Suppose $k$ and $l$ are two nodes, not necessarily distinct, with degrees greater than $2$ joined by a NB chain $k=i_{0},i_{1},\ldots,i_{p}=l$ of length $p$.
If $\mathbf{v}$ is nonzero on at least one of the edges $i_{j}\to i_{j+1}$, then it is nonzero on all of them, and in that case we have $\lambda^{2p}=1$. 
\end{lemma}

\begin{proof}
Since $\mathbf{v}$ is an eigenfunction of unitary eigenvalue, it satisfies $\mathbf{B^{*}B} \mathbf{v} = \mathbf{BB^{*}} \mathbf{v} = \mathbf{v}$.
By Lemma \ref{lem:bbt}, this implies $\vintok{v}{k} = \vintok{v}{l} = 0$ since $d_{k},d_{l}>2$.
To fix ideas, suppose $k$ and $l$ are neighbors, i.e. $p=1$.
The fact that $\vintok{v}{k} = 0$ and \cref{eqn:evector} give $\lambda \mathbf{v}_{k \to l} + \mathbf{v}_{l \to k} = 0 = \lambda \mathbf{v}_{l \to k} + \mathbf{v}_{k \to l}$.
Multiply the first equation by $\lambda$ and replace in the second to obtain $\left(\lambda^{2}-1\right)\mathbf{v}_{k\to l}=0$.
Since $\mathbf{v}_{k\to l}$ is nonzero we have $\lambda^{2}=1$.

Now assume $p>1$. Similarly to above, we have 
\begin{equation}\label{eqn:key-1}
\lambda\mathbf{v}_{k\to i_{1}}+\mathbf{v}_{i_{1}\to k}=0\quad\text{and}\quad\lambda\mathbf{v}_{l\to i_{p-1}}+\mathbf{v}_{i_{p-1}\to l}=0.
\end{equation}
Further, since $d_{i_{j}}=2$ for each $j=1,\ldots,p-1$, we have the following $2 \left( p - 1 \right)$ equations
\begin{equation}\label{eqn:key-2}
\lambda\mathbf{v}_{i_{j}\to i_{j+1}}=\mathbf{v}_{i_{j-1}\to i_{j}}\quad\text{and}\quad\lambda\mathbf{v}_{i_{j}\to i_{j-1}}=\mathbf{v}_{i_{j+1}\to i_{j}}.
\end{equation}
Multiplying $\lambda \mathbf{v}_{l \to i_{p-1}} + \mathbf{v}_{i_{p-1} \to l} = 0$ by $\lambda^{2p-1}$ and using eqs. \eqref{eqn:key-2} iteratively yields ${ \left( \lambda^{2p} - 1 \right) \mathbf{v}_{l \to i_{p-1}} = 0 }$.
But eqs. \eqref{eqn:key-2} and the fact that $\mathbf{v}$ is nonzero in at least one of the edges $i_j \to i_{j+1}$ imply that $\mathbf{v}_{l \to i_{p-1}} \neq 0$.
Thus we have $\lambda^{2p} = 1$.
\end{proof}


\begin{theorem}\label{thm:roots-of-unity}
Let $G$ be a graph with NB matrix $\mathbf{B}$, and suppose $\mathbf{B v} = \lambda \mathbf{v}$ with $\left| \lambda \right| = 1$.
Then, $\lambda$ is a root of unity.
\end{theorem}

\begin{proof}
Let $H$ be the support of $\mathbf{v}$.
By Lemma \ref{lem:support-md2}, $H$ has minimum degree at least $2$.
If $H$ is a circle graph, say of length $r$, then immediately every NB eigenvalue is a $r^{th}$ root of unity, as pointed out in \Cref{rmk:circle}.
If, on the other hand, $H$ contains at least two cycles, then it must contain at least one NB chain.
In that case, let $p$ be the length of some NB chain in $H$.
By Lemma \ref{lem:path-degree-2}, we have $\lambda^{2p} = 1$.
\end{proof}

The support of eigenfunctions of real unitary eigenvalue is characterized by Corollary \ref{cor:bases}.
The following result characterizes the nonreal case.

\begin{theorem}\label{thm:support}
Let $G$ be a graph with at least two cycles, with NB matrix $\mathbf{B}$, and suppose $\mathbf{B v} = \lambda \mathbf{v}$ with unitary nonreal $\lambda$.
Let $p^*$ be the smallest positive integer such that $\lambda^{p^*} = 1$.
Let $H$ be the support of $\mathbf{v}$.
Then, each edge of $H$ must be part of a unique NB chain and the length of any NB chain in $H$ is a multiple of $p^*$.
\end{theorem}

\begin{proof}
Note Theorem \ref{thm:roots-of-unity} guarantees that $p^*$ exists and the assumption that $\lambda$ is nonreal implies $p^* > 2$.

Now take an edge $x \to y$ in $H$.
If both $x$ and $y$ have degrees greater than $2$, then $x \to y$ is a NB chain of length $1$.
Otherwise, if $y$ has degree exactly $2$, let $z \neq x$ be its other neighbor and consider the chain $x \to y \to z$.
Now repeat this process extending the chain by one edge in either direction until both endpoints have degrees greater than $2$; this is always possible since $G$ is not a circle graph.
This process defines the unique NB chain that contains $x \to y$.

Let $\mathcal{P}$ be the set of all lengths of NB chains in $H$.
Theorem \ref{thm:roots-of-unity} guarantees that each such length must be greater than $1$, because otherwise we would have $\lambda^2 = 1$.
Take two different lengths $p$ and $p'$.
Again by Theorem \ref{thm:roots-of-unity} we have $\lambda^{2p} = 1 = \lambda^{2p'}$.
But $\lambda^2 \neq 1$ and thus $p$ and $p'$ cannot be coprime, i.e., one must divide the other.
Repeating the same argument for each pair of lengths, and letting $\hat{p}$ be the maximum common divisor of $\mathcal{P}$, we must have $\lambda^{\hat{p}} = 1$ and therefore $\hat{p}$ is a multiple of $p^*$.
\end{proof}

Theorem \ref{thm:support} establishes that the graph induced by the support of a NB eigenfunction must be comprised entirely of NB chains all of whose lengths are multiples of the same number.
These graphs are otherwise known as \emph{graph subdivisions}, and are the subject of the next section.

\section{Graph subdivisions }\label{sec:subdivisions}

\begin{definition}[\bf{Graph subdivision}]
Given a graph $G$ (possibly with multi-edges or self-loops) its \emph{$r$-subdivision} is the graph created by inserting $r-1$ nodes in each edge of $G$.
In other words, each edge $(k,l)$ in $G$ is replaced by a chain $k=i_{0},i_{1},\ldots,i_{r}=l$ where each $i_{j}$ has degree $2$ for $j=1,\ldots,r-1$.
Further, note that if $G$ is the $r$-subdivision of $G'$ and $G'$ is in turn the $r'$-subdivision of some $G''$, then $G$ is the $\left( r \times r' \right)$-subdivision of $G''$.
\end{definition}

Consider a graph $H$ that is the $p$-subdivision of some other graph $G$.
If $G$ has no nodes of degree less than $3$, then every NB chain in $H$ has length exactly $p$.
If $G$ has one node of degree $2$, then $H$ has one NB chain of length exactly $2p$.
In general, every edge in $H$ is contained in a unique NB chain and the length of this chain must be a multiple of $p$.
See \cref{fig:chain_examples}.
In light of this observation, Theorem \ref{thm:support} can be rephrased as follows.

\begin{figure}
\centering
\includegraphics[width=0.75\textwidth]{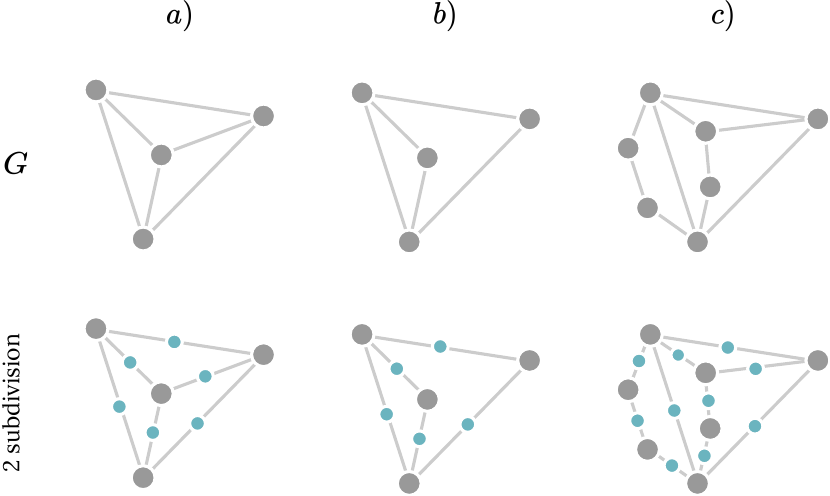}
\caption{\label{fig:chain_examples}
\textbf{NB chains and graph subdivisions.}
A graph $G$ (\textbf{top}) and its $p$-subdivision (\textbf{bottom}) for the case $p=2$.
The length of every chain in the subdivision is a multiple of $p$.
\textbf{a)} If the graph has no nodes of degree $2$, then every chain in the subdivision has length $p$.
\textbf{b)} Since $G$ has nodes of degree $2$, its subdivision has chains of length $2p$.
\textbf{c)} In general, the maximum length of a NB chain in the subdivision is $p$ times the maximum length of a NB chain in the original graph $G$.
}
\end{figure}

\begin{corollary}[\textbf{of Theorem \ref{thm:support}}]
$H$ must be the $p^*$-subdivision of some other graph.\qed
\end{corollary}

For this reason, we dedicate this section to studying the NB spectrum of graph subdivisions.
Kotani and Sunada \cite{kotani} were first to note the following relationship involving graph subdivisions.
Let $H$ be the $r$-subdivision of $G$, and let $\mathbf{B}_H, \mathbf{B}_G$ be their NB matrices, respectively. 
Then,
\begin{equation}\label{eqn:zeta-subdiv}
\det\left(t \mathbf{I} - \mathbf{B}_H \right) = \det \left( t^{r} \mathbf{I} - \mathbf{B}_G \right).
\end{equation}
In other words, the NB eigenvalues of $H$ are exactly the $r^{th}$ roots of the NB eigenvalues of $G$.

The eigenfunctions of a subdivision can also be constructed from the eigenfunctions of the original graph and vice versa, as the next result shows.

\begin{proposition}\label{pro:subdiv-functions}
Let $H$ be the $r$-subdivision of a graph $G$ and suppose $\mathbf{v}$ is an eigenfunction of $H$ with unitary eigenvalue $\lambda$.
Let $k,l$ be neighbors in $G$, and let $k = k_{0}(l), k_{1}(l), \ldots, k_{r}(l) = l$ be the corresponding NB chain in $H$.
Define the vector $\mathbf{u}$ as 
\begin{subequations}
\begin{align}
\mathbf{u}_{k \to l} &\coloneqq \mathbf{v}_{k \to k_{1}(l)},\quad \text{for each \ensuremath{k} with \ensuremath{d_{k}>2},} \label{eqn:def-u-a} \\
\mathbf{u}_{k \to j} &\coloneqq \mathbf{u}_{i \to k} / \lambda^{r},\quad \text{for each \ensuremath{k} with \ensuremath{d_{k}=2} and distinct neighbors \ensuremath{i,j}.} \label{eqn:def-u-b}
\end{align}
\end{subequations}
Then, $\mathbf{u}$ is a NB eigenfunction of $G$ with eigenvalue $\lambda^{r}$.
We say that $\mathbf{u}$ is the \emph{vector subdivision}, or simply the \emph{subdivision}, of $\mathbf{v}$.
\end{proposition}

\begin{proof}
First we show that $\mathbf{u}$ satisfies 
\begin{equation*}
\lambda^{r}\mathbf{u}_{k\to l}=-\mathbf{u}_{l\to k},\quad\text{for each edge \ensuremath{k \to l} in \ensuremath{G} where \ensuremath{d_{k}>2}.}
\label{eqn:def-u-c} \tag{5.2c}
\end{equation*}
We do so in three cases.

\emph{Case 1:} Assume $k,l$ are neighbors with $d_{k},d_{l}>2$.
By \cref{eqn:def-u-a} we have $\mathbf{u}_{k \to l} = \mathbf{v}_{k \to k_{1}(l)} = -\lambda^{-1} \mathbf{v}_{k_{1}(l) \to k}$, where last equality comes from $\lambda \mathbf{v}_{k \to k_{1}(l)} + \mathbf{v}_{k_{1}(l) \to k} = \vintok{v}{k} = 0$ which holds since $\lambda$ is unitary.
Further, recall that each $k_{i}(l)$ has degree $2$ and thus
\begin{equation}
\mathbf{u}_{k\to l}=-\lambda^{-1}\mathbf{v}_{k_{1}\left(l\right)\to k}=-\lambda^{-1}\lambda^{-1}\mathbf{v}_{k_{2}\left(l\right)\to k_{1}\left(l\right)}=-\lambda^{-1}\lambda^{-2}\mathbf{v}_{k_{3}\left(l\right)\to k_{2}\left(l\right)}=\ldots=-\lambda^{-1}\lambda^{-r+1}\mathbf{v}_{l\to k_{r-1}\left(l\right)}.
\end{equation}
Since $d_{l}>2$, applying \cref{eqn:def-u-a} again yields the desired equation,
\begin{equation}
\mathbf{u}_{k\to l}=-\lambda^{-r}\mathbf{v}_{l\to k_{r-1}\left(l\right)}=-\lambda^{-r}\mathbf{u}_{l\to k}.
\end{equation}

\emph{Case 2:} Assume $k,l$ are neighbors with $d_{k}>2$ and $d_{l}=2$.
Let $i \neq k$ be the only neighbor of $l$.
In this case we use an auxiliary set of node labels: set $l_{q}(k) \coloneqq k_{r-q}(l)$ for $q=0,1,\ldots,r$. 
In this case, \cref{eqn:def-u-b} gives $\lambda^{r}\mathbf{u}_{l\to k}=\mathbf{u}_{i\to l}$. To fix ideas, assume $d_{i}>2$.
\Cref{eqn:def-u-a} yields 
\begin{equation}
\lambda^{r}\mathbf{u}_{l\to k}=\mathbf{u}_{i\to l}=\mathbf{v}_{i\to i_{1}\left(l\right)}.
\end{equation}
But each $d_{i_{j}\left(l\right)}=2$, thus
\begin{equation}
\lambda^{r}\mathbf{u}_{l\to k}=\mathbf{v}_{i\to i_{1}\left(l\right)}=\lambda\mathbf{v}_{i_{1}\left(l\right)\to i_{2}\left(l\right)}=\ldots=\lambda^{r-1}\mathbf{v}_{i_{r-1}\left(l\right)\to l}.
\end{equation}
But $d_{l}=2$ as well, so
\begin{equation}
\lambda^{r}\mathbf{u}_{l\to k}=\lambda^{r-1}\mathbf{v}_{i_{r-1}\left(l\right)\to l}=\lambda^{r-1}\lambda\mathbf{v}_{l\to l_{1}\left(k\right)}=\ldots=\lambda^{r-1}\lambda^{r}\mathbf{v}_{l_{r-1}\left(k\right)\to k}.
\end{equation}
Finish by noting $\lambda\mathbf{v}_{k \to l_{r-1}(k)} + \mathbf{v}_{l_{r-1}(k) \to k} = \vintok{v}{k} = 0$
and applying \cref{eqn:def-u-a} again,
\begin{equation}
\lambda^{r}\mathbf{u}_{l\to k}=\lambda^{2r-1}\mathbf{v}_{l_{r-1}\left(k\right)\to k}=-\lambda^{2r-1}\lambda\mathbf{v}_{k\to l_{r-1}\left(k\right)}=-\lambda^{2r}\mathbf{u}_{k\to l},
\end{equation}
as desired.
If $d_i = 2$, we let $i'$ be the node of degree greater than $2$ closest to $l$ other than $k$.
Applying the previous argument inductively to each node between $l$ and $i'$ deals with this latter case.

\emph{Case 3:} Assume $k,l$ are neighbors with $d_k = d_l = 2$.
This case is handled in a similar way by induction on the number of nodes between $k$ and $l$ and the two nodes closest to them with degree greater than $2$.

Now we proceed to show that $\mathbf{u}$ is an eigenfunction of $G$ with eigenvalue $\lambda^{r}$.
Let $G$ have NB matrix $\mathbf{B}$ and $k$ have degree $d_{k}=2$ with distinct neighbors $l$ and $j$.
We have, by \cref{eqn:def-u-b},
\begin{equation}
\left(\mathbf{B u}\right)_{k \to l} = \vintok{u}{k} - \mathbf{u}_{l \to k} = \mathbf{u}_{j \to k} = \lambda^{r}\mathbf{u}_{k \to l},
\end{equation}
as desired.
Now suppose $d_{k}>2$.
By \cref{eqn:def-u-c} we have
\begin{equation}
\left(\mathbf{B u}\right)_{k \to l} = \vintok{u}{k} - \mathbf{u}_{l \to k} = \vintok{u}{k} + \lambda^{r} \mathbf{u}_{k \to l}.
\end{equation}
Showing $\vintok{u}{k} = 0$ finishes the proof.
Indeed, we have
\begin{equation}
\vintok{u}{k} = \sum_{l} \mathbf{u}_{l \to k} = -\lambda^{r} \sum_{l} \mathbf{u}_{k \to l} = -\lambda^{r} \sum_{l} \mathbf{v}_{k \to k_{1}(l)} = \left( -\lambda^r \right) \vfromk{v}{k} = 0,
\end{equation}
where we have used \cref{eqn:def-u-c,eqn:def-u-a} and the fact that $\mathbf{v}$ is non-leaky in $H$.
\end{proof}

\begin{remark*}
Proposition \ref{pro:subdiv-functions} builds an eigenfunction of the original graph $G$ from an eigenfunction of the subdivision $H$.
The converse is also possible, by applying the same procedure in reverse.
Furthermore, this construction preserves linear (in)dependence.
In other words, the eigenfunctions of $G$ uniquely determine those of $H$ and vice versa.
\end{remark*}

We can now characterize some of the NB eigenvalues of graph subdivisions as follows.

\begin{theorem}\label{thm:subdiv}
Given a graph $G$ with at least two cycles, $n$ nodes, and $m$ edges, suppose $H$ is its $r$-subdivision.
Let $\lambda$ be a $r^{th}$ root of unity.
Then, $\lambda$ is an eigenvalue of $H$ and it satisfies $\am_H(\lambda) = \gm_H(\lambda) = m - n + 1$.
Further, let $\mu$ be a $r^{th}$ root of $-1$.
Then each $\mu$ is an eigenvalue of $H$ and it satisfies $\am_H(\mu) = \gm_H(\mu) = m - n + \mathbf{1}\{ G \text{ is bipartite}\}$.
\end{theorem}

\begin{proof}
Recall $1$ is always a NB eigenvalue of $G$ and \cref{eqn:zeta-subdiv} implies $\lambda$ is an eigenvalue of $H$.
Moreover, \cref{eqn:zeta-subdiv} also implies $\am_H(\lambda) = \am_G(1)$.
Using Proposition \ref{pro:subdiv-functions}, given an eigenfunction $\mathbf{u}$ of $G$ of eigenvalue $1$, we can construct a unique eigenfunction $\mathbf{v}$ of $H$ of eigenvalue $\lambda$, and vice versa.
In other words, we also have $\gm_H(\lambda) = \gm_G(1)$.
The fact $\am_H(\lambda) = \gm_H(\lambda) = m - n + 1$ then follows from the fact that $\am_G(1) = \gm_G(1) = m - n + 1$, which was established in Propositions \ref{pro:am-plus1} and \ref{pro:gm-plus1}.

Since $G$ contains at least two cycles, it admits $-1$ as an eigenvalue by Proposition \ref{pro:gm-minus1}.
By an argument similar to that of the previous paragraph, $\mu$ is an eigenvalue of $H$ that satisfies $\am_H(\mu) = \am_G(-1) = \gm_G(-1) = \gm_H(\mu)$.
By Propositions \ref{pro:am-minus1} and \ref{pro:gm-minus1}, the value of this expression is $m - n + 1$ is $G$ is bipartite and $m - n$ otherwise.
\end{proof}

\begin{remark*}
In general, graph subdivisions may have unitary eigenvalues not covered by the preceding Theorem.
Namely, if $H$ is the $r$-subdivision of $G$ and $G$ admits a nonreal unitary eigenvalue $\lambda$, then each $r^{th}$ root of $\lambda$ is an eigenvalue of $H$.
\end{remark*}

Up to this point we have studied the setting where a graph $G$ has at least two cycles and contains an eigenvalue $\lambda$ with $\lambda^r = 1$.
We now know that the support of each eigenfunction of $\lambda$ must be a subgraph $H$ that is the $r$-subdivision of some other graph.
Furthermore, each such subgraph $H$ is actually the support of many eigenfunctions for each $r^{th}$ root of unity, as shown in Theorem \ref{thm:subdiv}.
\Cref{sec:cycle-graphs} will study the same setting but in the case where $G$ has exactly one cycle (since we are assuming minimum degree at least $2$, this means $G$ is precisely a circle graph).
\Cref{sec:gluing} determines how these subgraphs $H$ may be connected to the rest of the graph $G$.

\subsection{The case of circle graphs}\label{sec:cycle-graphs}

Let $G$ be a circle graph of length $r$.
As noted in \Cref{rmk:circle}, the NB eigenvalues of $G$ are precisely the $r^{th}$ roots of unity, each with multiplicity $2$.
This is seen by decomposing its NB matrix $\mathbf{B}$ in two permutation matrices.
The next example shows an alternative way of obtaining the same result.

\begin{example}\label{exm:circle-as-subdiv}
Let $O$ be the graph with one node and one self-loop.
The NB matrix of $O$ is the $2\times2$ identity matrix.
Note the circle graph $G$ of length $r$ is precisely the $r$-subdivision of $O$.
\Cref{eqn:zeta-subdiv} then shows that the eigenvalues of $G$ are exactly the $r$th roots of unity and each has multiplicity $2$.
\end{example}

In view of this observation, everything we have shown about graph subdivisions applies as well to circle graphs.
For example, the support of each NB eigenfunction is always the entirety of $G$, and $G$ can be considered as a pathological case of a NB chain: a closed chain whose endpoint has degree $2$.

What makes circle graphs special, and the reason why the previous sections do not consider circle graphs, is that circle graphs possess special NB eigenfunctions that other graphs do not.

\begin{lemma}
Let $G$ be a circle graph of length $r$ with a node $k$, and let $\lambda$ be one of the $r^{th}$ roots of unity.
Define the subspaces
\vspace{-0.75em}
\begin{align}
V &= \left\{ \mathbf{v} \colon \mathbf{B v} = \lambda \mathbf{v} \right\}, \nonumber \\
W(k) &= \left\{ \mathbf{v} \colon \mathbf{B v} = \lambda \mathbf{v} \text{ and } \vintok{v}{k} = 0 \right\}.
\end{align}
\vspace{-0.75em}
Then, we have $\dim V = 2$ and $\dim W(k) = 1$, for any $k$.
\end{lemma}

\begin{proof}
Note $V$ is the eigenspace corresponding to $\lambda$, and we have already discussed its dimension is $2$.
Now choose an arbitrary node $k$ and write $W \coloneqq W(k)$.
Label the nodes of $G$ consecutively as $k = i_1, \ldots, i_r$ and consider the following two functions.
\begin{align}
\mathbf{v}_{i_j \to i_{j+1}} = \lambda^{-j}, & \quad \mathbf{v}_{i_{j+1} \to i_j} = 0, \nonumber \\
\mathbf{u}_{i_j \to i_{j1+1}} = 0 \phantom{^{-j}}, & \quad \mathbf{u}_{i_{j+1} \to i_j} = \lambda^{-j},
\end{align}
where $j$ is taken modulo $r$. 
Elementary calculations show the following claims, which finish the proof.

\begin{itemize}
\item $\{\mathbf{u}, \mathbf{v}\}$ is a basis of $V$,
\item $\mathbf{u},\mathbf{v} \notin W$,
\item the only linear combinations of $\mathbf{u},\mathbf{v}$ that lie in $W$ are multiples of $\mathbf{u} - \mathbf{v}$. \qedhere
\end{itemize}
\end{proof}

In contrast, in graphs with two or more cycles, $W(k)$ is generally equal to $V$, for example when $d_k > 2$.
This will be of utmost importance in the following sections, as we will make extensive use of the set of nodes $k$ which satisfy $\vintok{v}{k} = 0$ for a fixed function $\mathbf{v}$.
Due to this observation, in what follows we will always separate our results based on whether $G$ is a circle graph or not.

\section{Gluing non-leaky functions}\label{sec:gluing}
Consider a graph $G$ that is the $r$-subdivision of some other graph, and consider an arbitrary graph $H$.
How may we connect, or \emph{glue}, $G$ to $H$ so that the resulting graph is such that i) it contains $G$ as a subgraph, and ii) it admits $r^{th}$ roots of unity as NB eigenvalues?
This setting may be considered a converse to the setting studied in previous sections where we started by considering a graph that admits a root of unity as eigenvalue and concluded that it must contain a subgraph that is an $r$-subdivision.

The following construction formalizes the notion of \emph{gluing} alluded to in the previous paragraph.

\begin{definition}[\bf{Non-leaky gluing}]\label{def:non-leaky-gluing}
Let $G,H$ be two graphs with node sets $V(G),V(H)$ and edge sets $E(G),E(H)$, respectively, and suppose $\mathbf{v}$ is a NB eigenfunction of $G$.
Now define the set
\begin{equation}
N \coloneqq N(G,\mathbf{v}) = \left\{k\in V(G):\,\vintok{v}{k} = 0\right\}.
\end{equation}
By definition, $\mathbf{v}$ does not leak via any of the nodes in $N$.
Now construct a graph $X$ by identifying (or \emph{gluing}) the
nodes in $N$ with nodes in $H$.
More formally, let $\sim$ be an equivalence relation in $V(G) \coprod V(H)$, where $\coprod$ denotes disjoint union.
In this construction we only consider equivalence relations $\sim$ such that each equivalence class is either trivial or it contains exactly one node from $N$ and one node from $V(H)$.
Define $X$ as having node set $V(X)=\left(V(G)\coprod V(H)\right)/\!\sim$~ and edge set $E(X)=E(H)\coprod E(G)$.
The equivalence relation $\sim$ is called a \emph{non-leaky gluing equivalence relation}, or simply a \emph{non-leaky gluing}.
See \Cref{fig:detach_glue} for examples. 
\end{definition}

\begin{figure}
\centering
\includegraphics[width=\textwidth]{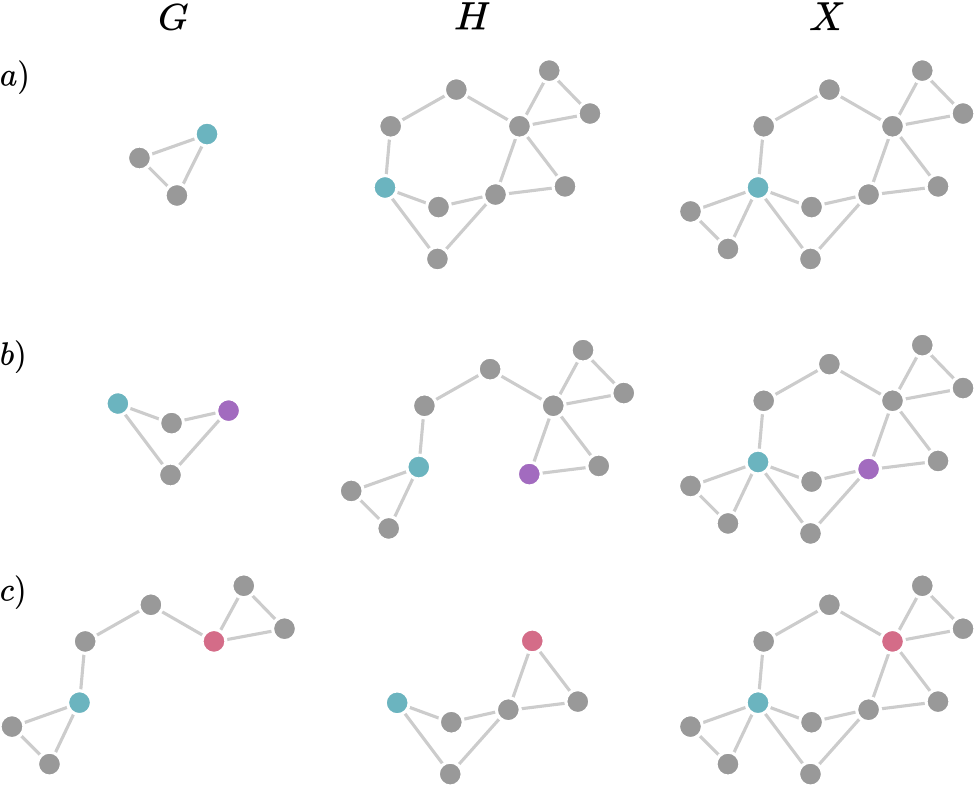}
\caption{\label{fig:detach_glue}
\textbf{Non-leaky gluing and detaching.}
For each row, gluing $G$ to $H$ by identifying the nodes by color, except for the gray nodes, produces the graph $X$.
The colored nodes in the left column comprise (a subset of) the set $N \subset V(G)$.
The three rows produce the same graph $X$ via different gluing operations.
Conversely, when analyzing $X$, one can detach a subgraph $G$ to study the (non-leaky) eigenfunctions supported on it.
\textbf{a)} Gluing an odd-length circle.
The gluing equivalence relation in this case has one equivalence class comprised of the two green nodes, one in $G$ and one in $H$.
\textbf{b)} Gluing an even-length circle.
The gluing has two non-trivial equivalence classes: one containing the two purple nodes and another containing the two green nodes.
\textbf{c)} Gluing a graph with two cycles.
The gluing has two non-trivial equivalence classes: one containing the two pink nodes and another containing the two green nodes.
\textbf{a-c)} In each row, the gluing is such that each gray node belongs to a trivial equivalence class, i.e. a class that contains a single element.
}
\end{figure}

\begin{remark*}
In what follows, it is a technical necessity that at least one of the graphs $G$, $H$ in Definition \ref{def:non-leaky-gluing} has minimum degree at least $2$ and that the non-leaky gluing $\sim$ is chosen in such a way that the resulting graph $X$ also has minimum degree at least $2$.
However, we will not need this level of generality and instead will require, for the sake of simplicity, that both $G$ and $H$ have minimum degree at least $2$, which guarantees that $X$ has minimum degree at least $2$, regardless of the gluing $\sim$. 
\end{remark*}

\begin{theorem}\label{thm:non-leaky-gluing}
Let $G,H$ be two graphs with minimum degree at least $2$ and suppose $\mathbf{v}$ is a NB eigenfunction of $G$ with unitary eigenvalue $\lambda$.
Consider the set $N(G, \mathbf{v})$ as in Definition \ref{def:non-leaky-gluing}, and let $\sim$ be a non-leaky gluing.
Construct $X$ by gluing $G$ to $H$ using $\sim$.
Then, $\lambda$ is a NB eigenvalue of $X$ with an eigenfunction $\mathbf{u}$ that is supported on (a copy of) $G$, i.e. $\mathbf{u}$ is equal to $\mathbf{v}$ in $E(G)\subset E(X)$ and zero in $E(H)\subset E(X)$.
We say $\lambda$ and $\mathbf{v}$ are \emph{transferred} to $X$ by gluing $G$.
\end{theorem}

\begin{proof}
Suppose we prove that $\mathbf{u}$ is non-leaky in $X$.
Then, since it is a non-leaky eigenfunction whose support is isomorphic to $G$ with eigenvalue $\lambda$, it is also an eigenfunction of $X$ with eigenvalue $\lambda$.
Thus it suffices to show that $\mathbf{u}$ is non-leaky in $X$.

Let $M \subset N(G, \mathbf{v})$ be the set of nodes that are glued by $\sim$, that is, the set of nodes that belong to a non-trivial equivalence class.
Note the node set of $X$ can be partitioned in three disjoint sets:
\begin{equation}\label{eqn:glue-node-partition}
V(X) \,\, = \,\, \left( V(G) \setminus M \right) \,\, \coprod \,\, \left(V(H) \setminus M \right) \,\, \coprod \,\, M.
\end{equation}
We proceed by proving that each node $k \in V(X)$ satisfies $\left(d_{k}-2\right)\vintok{u}{k} = 0$, depending on which set of the partition it belongs to.
These cases are illustrated in \cref{fig:gluing_proof_cases}.

\emph{Case 1:} Let $k \in V(G) \setminus M$.
By construction, every edge $l \to k$ incident to $k$ is an edge of $G$.
Since $\mathbf{u}$ equals $\mathbf{v}$ on all of $E(G)$, we have $\mbox{\ensuremath{\left(d_{k}-2\right)\vintok{u}{k}=\left(d_{k}-2\right)\vintok{v}{k}=0}}$
because $\mathbf{v}$ is non-leaky in $G$.

\emph{Case 2:} Let $k \in V(H) \setminus M$.
The only edges $l \to k$ incident to $k$ are edges of $H$.
Since $\mathbf{u}$ is zero in $E(H)$, we have ${\ensuremath{\vintok{u}{k}}} = 0$.

\emph{Case 3:} Let $k \in M$.
We can partition the edges incident to $k$ as those that belong to $E(G)$ and those that belong to $E(H)$.
Then we have
\begin{equation}
\vintok{u}{k} = \sum_{i \to k \in E(G)} \mathbf{u}_{i \to k} + \sum_{j \to k \in E(H)} \mathbf{u}_{j \to k} = \vintok{v}{k} + 0 = 0,
\end{equation}
where the second equation uses again the fact that $\mathbf{u}$ equals $\mathbf{v}$ in $E(G)$ and equals zero in $E(H)$.
\end{proof}

\begin{figure}[t]
\centering
\includegraphics[width=\textwidth]{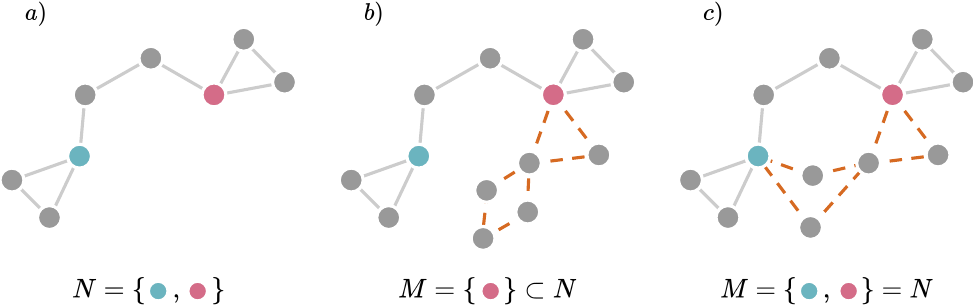}
\caption{\label{fig:gluing_proof_cases}
\textbf{Edges and nodes after gluing.}
\textbf{a)} A graph $G$ with its corresponding set $N=N(G, \mathbf{v})$ where $\mathbf{v}$ is such that $\mathbf{B v} = \lambda \mathbf{v}$ and $\lambda$ is a nonreal third root of unity.
\textbf{b)} The graph $H$ (with dashed edges) is glued to $G$ (with solid edges) by identifying a single node.
\textbf{c)} The same graph $H$ is glued to $G$ by identifying two nodes.
The proof of Theorem \ref{thm:non-leaky-gluing} considers nodes in three cases: those nodes that are only incident to solid edges, those nodes that are only incident to dashed edges, and those nodes that are incident to both.
Nodes in the third case are precisely those nodes that have been glued.}
\end{figure}

\begin{remark*}
The preceding proof holds regardless of what $\sim$ is.
In other words, the eigenfunction is transferred regardless of how the graph is glued, as long as $\sim$ is a valid non-leaky gluing.
Note also that $M$ may be a proper subset of $N(G,\mathbf{v})$.
Indeed, one need not glue \emph{every} possible node.
\end{remark*}

If $G$ in Theorem \ref{thm:non-leaky-gluing} is a subdivision, then we have shown that non-leaky gluing is the answer to the question at the start of this section.
That is, non-leaky gluing describes exactly all possible ways of connecting a graph subdivision ($G$) to another graph ($H$) so that the resulting graph ($X$) admits the prescribed unitary eigenvalues.
Indeed, trying to glue $G$ in any other way would, by definition, cause the transferred functions to leak and thus no longer be eigenfunctions of $X$.

\subsection{Characterizing the set \texorpdfstring{$N(G, \mathbf{v})$}{N(G,v)}}\label{sec:characterizing}
This section is dedicated to characterizing the set $N(G, \mathbf{v})$ when $G$ is a subdivision.
First, note that when $\lambda$ is a real root of unity, \cref{eqn:system-1,eqn:system-minus-1} imply that every node is in $N(G,\mathbf{v})$, for any eigenfunction $\mathbf{v}$.
For the nonreal case, we first prove a technical Lemma, which is essentially a continuation of the Key Lemma \ref{lem:path-degree-2}.

\begin{lemma}\label{lem:nodes-in-n}
Let $G$ be a graph with NB eigenfunction $\mathbf{v}$ of nonreal unitary eigenvalue $\lambda$. 
Suppose $k,l$ are nodes in $G$ with degrees greater than $2$ and such that $k = i_0, i_1, \ldots, i_p = l$ is a NB chain of length $p$.
Additionally, suppose $\lambda$ has order $q$, that is, $\lambda^q = 1$ and $\lambda^{q'} \neq 1$ for any $0 < q' < q$.
If $\mathbf{v}$ is nonzero on at least one of the edges of the chain, then 
\begin{equation}\label{eqn:nodes-in-n}
\left\{i_j \colon 0 \leq j \leq p \text{ and } 2j \text{ is a multiple of } q \right\} \subset N(G,\mathbf{v}),
\end{equation}
and no other node in the chain is in $N(G,\mathbf{v})$.
\end{lemma}

\begin{proof}
We may assume without loss of generality that $\mathbf{v}_{k \to i_1} = 1$.
Observe that eqs. \eqref{eqn:key-2} from the Key Lemma \ref{lem:path-degree-2} also apply in the present situation.
We may use these equations to find
\begin{equation}
\mathbf{v}_{i_{j}\to i_{j+1}} = \lambda^{-j}, \text{ for each }j=0,\ldots,p-1.
\end{equation}
This, plus the fact that $\vintok{v}{l} = 0$, yields $\mathbf{v}_{i_p \to i_{p-1}} = -\lambda^{-p}$ and therefore
\begin{equation}
\mathbf{v}_{i_{j} \to i_{j-1}} = - \lambda^{-2p + j} = - \lambda^{j}, \text{ for each }j=1,\ldots,p,
\end{equation}
where the last equation uses the fact $\lambda^{2p} = 1$, implied by the Key Lemma \ref{lem:path-degree-2}.

Now let $j \in \{1,\ldots,p-1\}$.
Since $i_j$ has degree exactly $2$, we have
\begin{equation}
\vintok{v}{i_j} = \mathbf{v}_{i_{j-1} \to i_{j}} + \mathbf{v}_{i_{j+1} \to i_{j}} = \lambda^{-(j-1)} - \lambda^{j+1}.
\end{equation}
Therefore,
\begin{equation}\label{eqn:middle-node}
\vintok{v}{i_j} = 0 \iff \lambda^{2j} = 1.
\end{equation}
Since $\lambda$ is nonreal, \cref{eqn:middle-node} is non-trivial.
Note $j$ is a solution if and only if $2j$ is a multiple of $q$.
Thus
\begin{equation}\label{eqn:nodes-in-n-1}
\left\{i_j \colon 0 < j < p \text{ and } 2j \text{ is a multiple of } q \right\} \subset N(G,\mathbf{v}).
\end{equation}

Additionally, since the degrees of $k,l$ are larger than $2$ and $\mathbf{v}$ is non-leaky, we immediately have
\begin{equation}\label{eqn:nodes-in-n-2}
\vintok{v}{k} = \vintok{v}{l} = 0.
\end{equation}

Finally, the Key Lemma implies that $2p$ is a multiple of $q$.
This fact plus \cref{eqn:nodes-in-n-1,eqn:nodes-in-n-2} are equivalent to \cref{eqn:nodes-in-n}.
No other node is in the chain due to \cref{eqn:middle-node}, and the proof is finished.
\end{proof}

\begin{corollary}\label{lem:nodes-in-n-subdiv}
Under the assumptions of Lemma \ref{lem:nodes-in-n}, further suppose $\lambda$ has order $p$ or $2p$.
Then the following claims are true.
\begin{enumerate}
\item The endpoints of the chain are in $N$, i.e. $\{k,l\} \subset N(G, \mathbf{v})$.
\item If $\lambda$ has order $p$ and $p$ is even, then $i_{p/2} \in N(G,\mathbf{v})$.
\item No other node in the chain is in $N(G,\mathbf{v})$.
\end{enumerate}
\end{corollary}

\begin{proof}
Direct from Lemma \ref{lem:nodes-in-n} when $q \in \{p, 2p\}.$
\end{proof}

Having determined the nodes in a single NB chain that are members of $N(G, \mathbf{v})$, fully characterizing this latter set is merely a corollary.
For convenience, we introduce the following terminology.

\begin{definition}
Let $H$ be a subdivision of $G$.
There is a natural inclusion from the node set of $G$ to the node set of $H$, up to relabelling of the nodes.
A node in the image of this inclusion is called a \emph{true node}.
A node of $H$ that is not a true node is called a \emph{subdivision node}.
Now suppose $H$ is the $r$-subdivision of $G$.
If $r$ is even, each NB chain in $H$ has a subdivision node that is exactly in the middle of it; we call this node the \emph{middle node}.
If $r$ is odd, there are no middle nodes.
See Figure \ref{fig:true_nodes}.
\end{definition}

\begin{figure}[t]
\centering
\includegraphics[width=\textwidth]{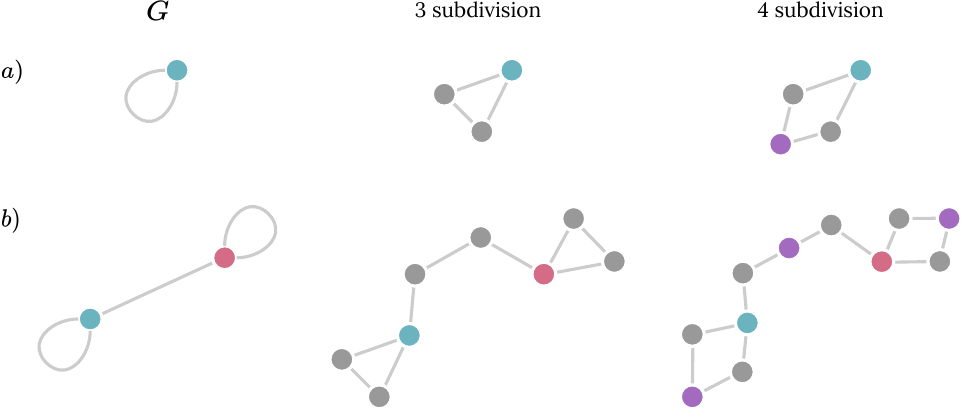}
\caption{\label{fig:true_nodes}
\textbf{True and middle nodes in graph subdivisions.}
A graph $G$ (left) and its $3$-subdivision (middle) and $4$-subdivision (right).
\textbf{a)} When $G$ is the graph $O$ with one node and one self-loop, its subdivisions are the circle graphs.
\textbf{b)} A graph with two cycles, cf. \cref{fig:detach_glue}(c). True nodes are color-coded.
Middle nodes are purple.
Subdivision nodes that are not middle nodes are  gray.
}
\end{figure}

\begin{proposition}[\bf{Gluing primitive roots in a subdivision}]\label{pro:glue-subdiv-primitive}
Let $G$ be a graph with minimum degree at least $2$ and with at least two cycles and suppose $G$ is the $r$-subdivision of some graph $H$.
Let $\mathbf{v}$ be an eigenfunction of $G$ corresponding to eigenvalue $\lambda$ such that $\lambda$ has order exactly $r$.
Then, 
\begin{equation}
N(G,\mathbf{v})=\{u\in V(G):\text{\ensuremath{u} is a true node or a middle node}\}.
\end{equation}
\end{proposition}

\begin{proof}
The result follows from applying Corollary \ref{lem:nodes-in-n-subdiv} in turn to each NB chain of $G$, and the observation that each true node is always the endpoint of some NB chain.
\end{proof}

\begin{proposition}[\bf{Gluing arbitrary roots in a subdivision}]\label{pro:glue-subdiv}
Let $G, r, H$ and $\mathbf{v}$ be as in Proposition \ref{pro:glue-subdiv-primitive}.
Additionally, suppose $\lambda$ has order $q$, not necessarily equal to $r$.
Then, each true node of $G$ is in $N(G, \mathbf{v})$.
A subdivision node is in $N(G, \mathbf{v})$ if and only if its distance to the nearest true node is a multiple of $q/2$, as long as said distance is an integer.
There are no other nodes in $N(G, \mathbf{v})$.
\end{proposition}

\begin{proof}
Direct from applying Lemma \ref{lem:nodes-in-n} to each NB chain of $G$ in turn.
\end{proof}

The case when $G$ is a circle graph is essentially the same though the proof is more subtle, due to the discussion in \cref{sec:cycle-graphs}.

\begin{proposition}[\bf{Gluing a circle}]\label{pro:glue-cycle}
Suppose $G$ is a circle graph of length $r$ and fix one of its NB eigenvalues $\lambda$.
With the notations of \cref{sec:cycle-graphs}, let $\mathbf{v} \in W$ and let $w$ be the unique node such that $\vintok{v}{w} = 0$.
Then the following are true.
\begin{enumerate}
\item If $r$ is odd then $N(G,\mathbf{v}) = \{ w \}$.
\item If $r$ is even then $N(G,\mathbf{v}) = \{ w, w' \}$, where $w'$ is the node diametrically opposite $w$.
\item Among all eigenfunctions of $\lambda$, $\mathbf{v}$ is the only one that admits non-empty $N(G,\mathbf{v})$, up to scalar multiples and up to the choice of $w$.
\end{enumerate}
\end{proposition}

\begin{proof}
We may consider $w$ to be the two endpoints of a (closed) NB chain in $G$, and relabel the rest of the nodes as $w=i_0, i_1, \ldots, i_r = w$.
In the present situation, almost all assumptions of Lemma \ref{lem:nodes-in-n} are valid except for one exception: the endpoints of the stipulated NB chain do not have degrees greater than $2$ (in the present case, both endpoints are $w$).
However, by construction we have $\vintok{v}{w} = 0$.
This is in fact the only property of the endpoints that was used in Lemma \ref{lem:nodes-in-n} and we can therefore apply it here.
We conclude that $w \in N(G, \mathbf{v})$, that if $r$ is even then $i_{r/2} (= w')$ is also in $N(G, \mathbf{v})$, and that no other node is in $N(G, \mathbf{v})$.

The third claim is true because if an eigenfunction $\mathbf{u}$ satisfies $\vintok{u}{w} = 0$ then it must be a scalar multiple of $\mathbf{v}$.
Lastly, one may construct a different function $\mathbf{x}$ that satisfies $\vintok{x}{z}$ for some other node $z$.
In that case, the result is the same up to a relabeling of the nodes.
\end{proof}

\begin{remark*}
Example \ref{exm:circle-as-subdiv} argued that the circle graph $G$ of length $r$ is the $r$-subdivision of the graph $O$ with one node and one self-loop.
Thus, $G$ has a single true node and, if $r$ is even, a single middle node (up to relabelling of the nodes).
In this light, Propositions \ref{pro:glue-subdiv-primitive} and \ref{pro:glue-cycle} have equivalent conclusions.
\end{remark*}

We close this section by exhibiting some further properties of the set $N$.

\begin{corollary}\label{lem:def-ngq}
Let $G$ be a graph subdivision with at least two cycles.
Suppose $\lambda$ is a NB nonreal unitary eigenvalue of $G$ with order $q$, and let $\mathbf{v}_1,\mathbf{v}_2$ be two eigenfunctions of $\lambda$.
Furthermore, suppose $\mu \neq \lambda$ is another nonreal unitary eigenvalue also with order $q$ and with eigenfunction $\mathbf{u}$.
Then $N(G,\mathbf{u}) = N(G,\mathbf{v_1}) = N(G,\mathbf{v_2})$.
\end{corollary}

\begin{proof}
The proofs of Propositions \ref{pro:glue-subdiv-primitive}, \ref{pro:glue-subdiv}, and \ref{pro:glue-cycle} depend only on the order of primitivity of $\lambda$ and the length of the underlying chain and not on the actual value of $\lambda$ or on any other property of $\mathbf{v}$.
Since every edge in a graph subdivision is contained in exactly one NB chain, the claim follows.
\end{proof}

\begin{definition}\label{def:ngq}
If $G$ has at least two cycles and $\lambda$ is a unitary NB eigenvalue with $\lambda^q = 1$ and eigenfunction $\mathbf{v}$, define $N(G,q) \coloneqq N(G,\mathbf{v})$.
This is well defined due to Corollary \ref{lem:def-ngq}.
If $G$ has no eigenvalues satisfying $\lambda^{q} = 1$, we let $N(G,q) \coloneqq \emptyset$.
\end{definition}

\begin{corollary}\label{cor:all-transfer}
Let $G$ be a graph subdivision with at least two cycles and $q$ be a positive integer such that $N(G,q)$ is non-empty.
Let $X$ be the result of gluing $G$ to another graph $H$ by identifying some nodes in $N(G, q)$.
If $\lambda^q = 1$ is a NB eigenvalue of $G$ then all eigenfunctions of $\lambda$ are transferred to $X$.
\end{corollary}

\begin{proof}
Direct from Corollary \ref{lem:def-ngq}.
\end{proof}

\begin{corollary}\label{cor:all-transfer-odd}
Let $G$ be a graph subdivision with at least two cycles, and let $q,q'$ be two positive integers such that $q$ is a multiple of $q'$.
Then, $N(G, q) \subset N(G, q')$.
Furthermore, if $q$ is odd, then $N(G, 2q) = N(G, q)$.
\end{corollary}

\begin{proof}
The first claim is direct from Proposition \ref{pro:glue-subdiv}.
For the second claim, suppose $q$ is odd.
In this case, Proposition \ref{pro:glue-subdiv} states that for every node in $N(G, q)$, the distance to the nearest true node is in the set $\{0, q/2, q, 3q/2, 2q, \ldots\}$, while for every node in $N(G, 2q)$, the distance to the nearest true node is in the set $\{0, q, 2q, \ldots\}$.
But the distance must always be an integer, and we are done.
\end{proof}

\begin{remark*}
Results \ref{lem:def-ngq} and \ref{cor:all-transfer-odd} also apply to the case when $G$ is a circle graph, because the NB eigenvalues of the circle graph of length $r$ are the $r^{th}$ roots of unity, each with multiplicity $2$.
On the other hand, result \ref{cor:all-transfer} does not apply to circle graphs.
Indeed, result \ref{pro:glue-cycle} showed that in this case only one out of the two different eigenfunctions is transferred.
\end{remark*}

\begin{remark*}
Theorem \ref{thm:non-leaky-gluing} showed that a single eigenfunction is transferred to $X$ via non-leaky gluing.
Corollary \ref{cor:all-transfer} extends this result, for graphs with at least two cycles, to prove that all eigenfunctions of all eigenvalues of the same order $q$ are transferred at the same time.
Furthermore, when $q$ is odd, Corollary \ref{cor:all-transfer-odd} shows that all eigenfunctions of all eigenvalues with order $q$ and $2q$ are transferred at the same time.
\end{remark*}

\subsection{Partial support}\label{sec:partial-support}
Before we can make general claims about the geometric multiplicity in the nonreal case, we must first consider the following situation.
Construct $X$ by gluing $G$ and $H$ together.
The previous sections establish when the eigenfunctions of $G$ and $H$ are transferred to $X$.
However, it is possible that new eigenfunctions arise in $X$ that come neither from $G$ nor $H$.
This happens, for example, when there are subgraphs $G', H'$ of $G, H$, respectively, that serve as partial support of an eigenfunction.

\begin{figure}[t]
\centering
\includegraphics[width=\textwidth]{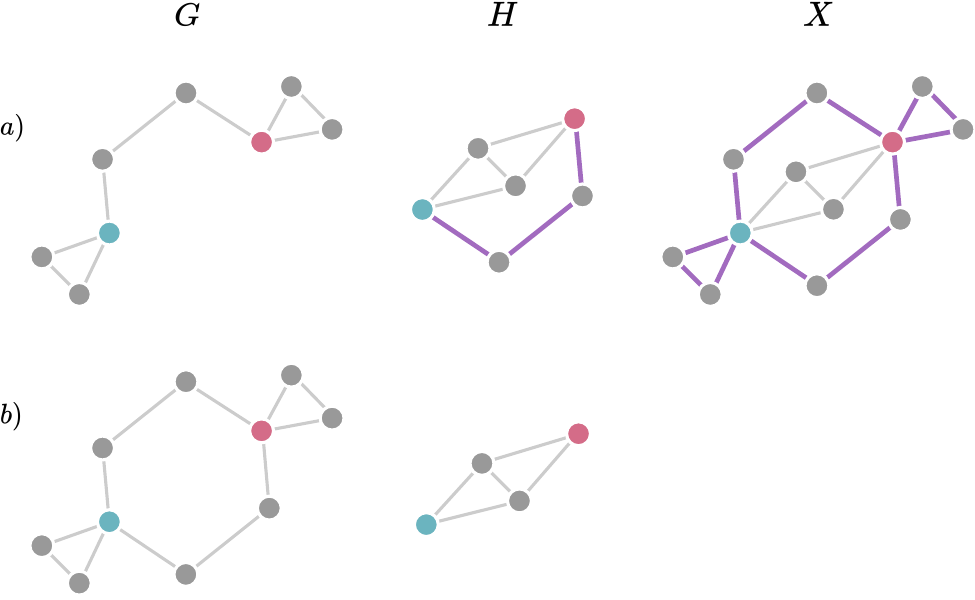}
\caption{\label{fig:partial-support}
\textbf{Example of partial support.}
See Example \ref{exm:partial-support}.
}
\end{figure}

\begin{example}\label{exm:partial-support}
Let $\lambda$ be a primitive sixth root of unity.
\Cref{fig:partial-support}a shows two graphs $G$ and $H$ that are glued by identifying the color-coded nodes to form a graph $X$.
Note $G$ is the $3$-subdivision of a (multi-)graph with one even cycle (cf. \cref{fig:true_nodes}b).
Therefore, we have $\gm_G(\lambda) = 1$.
The graph $H$ does not contain a subgraph that is a $3$- or $6$-subdivision and has minimum degree $2$, thus we have
$\gm_H(\lambda) = 0$.
The graph $X$ contains a subgraph that is a $3$-subdivision (far right, thick purple edges).
A straightforward computation shows that this purple subgraph is the $3$-subdivision of a graph that contains two linearly independent even cycles.
Therefore, $\gm_X(\lambda) = 2$.
Since $G$ and $H$ have been glued to form $X$, one of the eigenfunctions of $\lambda$ in $X$ can be identified with the unique eigenfunction of $\lambda$ in $G$.
The other eigenfunction of $X$ comes from the fact that $H$ contains a single NB chain of length $3$ (top row, middle, thick red edges).
This NB chain is such that, after the gluing, together with the copy of $G$ in $X$, it forms a larger subgraph that is a $3$-subdivision (far right, thick purple edges).
This NB chain in $H$ is what we call \emph{partial support}: it is a subgraph that does not support an eigenfunction, but when glued in certain ways to other graphs, it augments the maximal subgraph in $X$ that can support eigenfunctions.

\Cref{fig:partial-support}b shows an alternative pair of graphs $G$ and $H$ that, when glued by identifying the color-coded nodes, yield the same graph $X$ as in \cref{fig:partial-support}a.
In this case, however, there is no instance of partial support, as all the eigenfunctions in $X$ come from the eigenfunctions in $G$.
\end{example}

In the rest of the manuscript we will assume that gluings are made such that there is no partial support.
As illustrated example \ref{exm:partial-support}, a sufficient condition for this is the following.
Let $S$ be a maximal connected subgraph of $X$ that is a subdivision, and recall we can write $E(X) = E(G) \coprod E(H)$.
Then, we require that either $E(S) \subset E(G)$ or $E(S) \subset E(H)$.

\section{Geometric multiplicity}\label{sec:geometric}

Let $X$ be a connected graph possibly with multi-edges or self-loops, and let $\lambda$ be a root of unity of order $q$.
In this section we compute the geometric multiplicity of $\lambda$ as a NB eigenvalue of $X$.
It is clear by now that to do so, we must first find all NB chains in $X$ of the appropriate length and then assemble them in subgraphs of $X$ that may support an eigenfunction of $\lambda$.
In this section, by $H \subset G$ we mean $E(H) \subset E(G)$.

\begin{definition}
For a fixed positive integer $q>2$, define
\begin{equation}
\mathcal{C}_q \coloneqq \{C_q \subset G \colon C_q \text{ is a NB chain whose length is a multiple of } q \text{ or } q/2 \}.
\end{equation}
Note the case when the length of a chain is a multiple of $q/2$ is only relevant when $q$ is even.
Furthermore, let $\overline{\mathcal{C}}_q$ be the subgraph assembled from all such chains,
\begin{equation}
E \left( \overline{\mathcal{C}}_q \right) \coloneqq \bigcup_{C_q \subset \mathcal{C}_q} E(C_q).
\end{equation}
Note $\overline{\mathcal{C}}_q$ may be disconnected.
Therefore we define
\begin{equation}
\mathcal{S}_q \coloneqq \{S_q \subset \overline{\mathcal{C}}_q \colon S_q \text{ is a connected component of } \overline{\mathcal{C}}_q \}.
\end{equation}
\end{definition}

Each $S_q \in \mathcal{S}_q$ that contains at least two cycles supports at least one eigenfunction of $\lambda$.
Since each $S_q$ is maximal with respect to inclusion, there is no instance of partial support (see \cref{sec:partial-support}).

If $q$ is odd, we can directly use the graphs $S_q$ to compute the multiplicity of $\lambda$.
However, if $q$ is even, we need sometimes consider the graph of which $S_q$ is a subdivision.

\begin{definition}
For a graph $S$ that is the subdivision of some other graph, we write $\coll(S, q)$ for the graph whose $q$-subdivision is $S$, if it exists.
In that case, we say $\coll(S, q)$ is the \emph{collapse} of $S$.
\end{definition}

\begin{theorem}[\textbf{Main Theorem}]\label{thm:main}
Let $X$ be a connected graph with at least two cycles, possibly with multi-edges or self-loops, and let $\lambda$ be a root of unity of order $q>2$.
If $\mathcal{S}_q$ is empty, we have $\gm_X(\lambda) = 0$.
If $\mathcal{S}_q$ is non-empty and $q$ is odd,
\begin{equation}\label{eqn:main-odd}
\gm_X(\lambda) = \sum_{S_q \in \mathcal{S}_q} |E(S_q)| - |V(S_q)| + 1.
\end{equation}
If $\mathcal{S}_q$ is non-empty and $q$ is even,
\begin{equation}\label{eqn:main-even}
\gm_X(\lambda) = \sum_{\substack{S_q \in \mathcal{S}_q: \\ q | r(S_q)}} |E(S_q)| - |V(S_q)| + 1 \;\;+\; \sum_{\substack{S_q \in \mathcal{S}_q: \\ q \, \nmid \, r(S_q)}} |E(S_q)| - |V(S_q)| + \mathbf{1}\{ \coll(S_q, q) \textnormal{ is bipartite} \},
\end{equation}
where $r(S)$ is the maximum integer $r$ such that $S$ is the $r$-subdivision of some other graph.
\end{theorem}

\begin{proof}
We claim that each eigenfunction of $\lambda$ must be supported on $\mathcal{S}_q$.
Recall the components $S_q$ of $\mathcal{S}_q$ were assembled using chains whose length is a multiple of $q$ or $q/2$, which guarantees that the distance between every pair of nodes with degrees greater than $2$ is a multiple of $q/2$.
Thus, every such node in $S_q$ is in $N(S_q, q)$ by Theorems \ref{pro:glue-subdiv} and \ref{pro:glue-cycle}, which in turn means we may conceive of $S_q$ as having been glued to $X \setminus S_q$ to form $X$ by a valid non-leaky gluing.
Therefore, every eigenfunction of $\lambda$ in $S_q$ is transferred to $X$, and there are no others (other than linear combinations thereof).

Since each $S_q$ is a subdivision, we may use \Cref{thm:subdiv} to compute the number of eigenfunctions transferred.
If $q$ is odd, each $S_q$ is always a $q$-subdivision of some other graph.
Thus, in this case the number of eigenfunctions transferred is $|E(S_q)| - |V(S_q)| + 1$.

Now assume $q$ is even.
In this case, each $S_q$ may be a $q$-subdivision or not.
If it is not, then it must be a $q/2$-subdivision.
This is because $S_q$ is assembled from chains whose lengths devide $q$ or $q/2$.

If $S_q$ is a $q$-subdivision, or equivalently if $q | r(S_q)$, the $q$th roots of unity are eigenvalues of $S_q$ corresponding to the eigenvalue $+1$ in $\coll(S_q, q)$.
Again due to \Cref{thm:subdiv}, there are $E(\coll(S_q, q)) - V(\coll(S_q, q)) + 1$ of these eigenvalues.
This in turn equals $E(S_q) - V(S_q) + 1$, as taking subdivisions does not change the number of cycles in a graph.

If $S_q$ is not a $q$-subdivision, or equivalently if $q \nmid r(S_q)$, it must be a $q/2$-subdivision, and the $q$th roots of unity are eigenvalues of $S_q$ corresponding to the eigenvalue $-1$ in $\coll(S_q, q)$.
Using \Cref{thm:subdiv} again finishes the proof.
\end{proof}

\begin{remark}\label{rmk:main}
The Main \Cref{thm:main} does not claim that each $S_q$ is the support of an eigenfunction of $\lambda$.
This is in fact the case when $S_q$ has minimum degree at least $2$.
However, it is possible that some $S_q$ have nodes of degree $1$, for example if there is a ``dangling'' open NB chain.
In this case, the dangling chain is not part of the support of any eigenfunction, as the support must always have minimum degree at least $2$, see \Cref{lem:support-md2}.
Nevertheless, this case does not affect the formulas \eqref{eqn:main-odd} and \eqref{eqn:main-even} because such a dangling chain does not change the value of the expression $|E(S_q)| - |V(S_q)|$.
See the upcoming examples and \Cref{fig:examples}e for an illustration.
\end{remark}

\subsection{Examples}
\Cref{fig:examples} shows illustrative applications of the Main \Cref{thm:main} to several graphs.
For each row, the left-most column shows a graph $X$, while subsequent columns show all components $S_q$ for $q \in \{3,4,6,12\}$.
Each panel in each column (except for the first column) shows the geometric multiplicity of any unitary $\lambda$ with order $q$ (column) in $X$ (row). 
In each panel, gray nodes do not belong to any NB chain and golden nodes belong to multiple NB chains.
Other nodes are color-coded according to the unique NB chain they belong to.

\Cref{fig:examples}a shows a typical application of our results.
Note $X$ contains a copy of the complete graph on four nodes, $K_4$, glued to the golden node.
This subgraph could be replaced with any other graph $G$ and the results would be unchanged (unless of course $G$ contained further NB chains).

\Cref{fig:examples}b shows that a simple relocation of $K_4$ when compared to \Cref{fig:examples}a destroys many eigenfunctions but preserves others.
Even though both instances contain a cycle of length $12$, the chains that comprise the cycle have different lengths.

In \Cref{fig:examples}c, relocating $K_4$ again destroys some eigenfunctions and preserves others.
Note in this case $S_6$ is non-empty though it does not support any eigenfunctions for unitary eigenvalues of order $6$ because i) $q=6$ is even, ii) we have $r(S_6)=3$, iii) $q$ does not divide $r(S_6)$, and iv) $|E(S_6)| - |V(S_6)| = 0$.

The graph $X$ in \Cref{fig:examples}d has the same multiplicities as the graph $X$ in \Cref{fig:examples}a, even though the chains in these examples are different.
This is because even orders consider chains of length $q/2$.

\begin{figure}[!ht]
\centering
\includegraphics[width=0.985\textwidth]{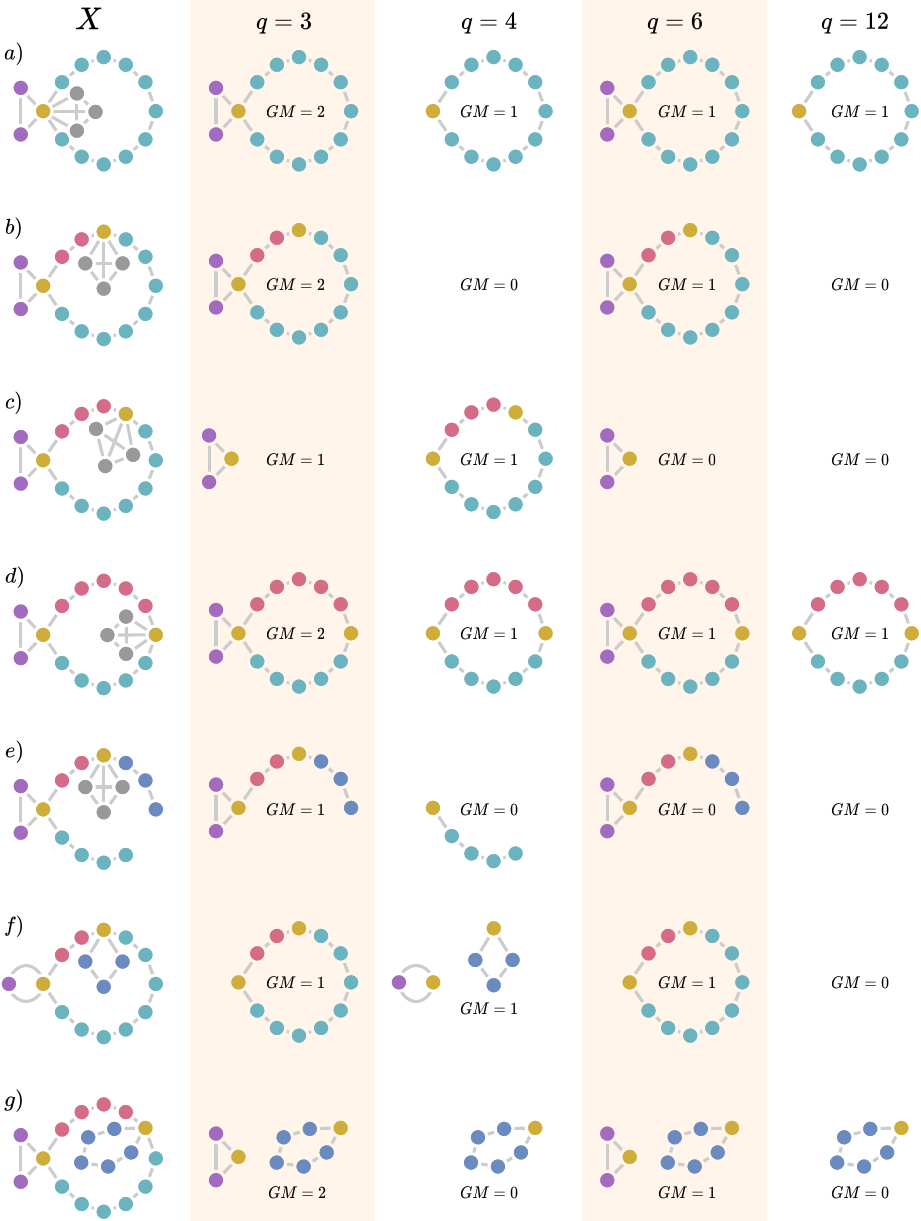}
\caption{\label{fig:examples}
Illustrative applications of the Main \Cref{thm:main}.}
\end{figure}

\Cref{fig:examples}e illustrates \Cref{rmk:main}.
Even when some $S_q$ contain nodes of degree $1$, the value of the expression $|E(S_q)| - |V(S_q)|$ never changes.

\Cref{fig:examples}f illustrates the case of multiple edges.
If two nodes share multiple edges and both nodes have degrees greater than $2$ (not pictured), then each of these multiple edges is not part of any NB chain with length greater than $1$ and thus does not affect the multiplicities of nonreal unitary eigenvalues.
If two nodes share exactly two edges, and furthermore one of those nodes has degree $2$, as pictured in \Cref{fig:examples}f, then we have a single closed NB chain of length $2$.
This chain appears as one of the components $S_4$.
However, this component does not support any eigenfunction since i) $q=4$ is even, ii) we have $r(S_4) = 2$, iii) $q$ does not divide $r(S_q)$, and iv) $|E(S_4)| - |V(S_4)| = 0$.
In this example, there is another component $S_4'$ which does support an eigenfunction. 
The case of self-loops is similar (not pictured for space reasons): each self-loop counts as a closed NB chain of length $1$ and therefore does not affect the multiplicities of nonreal unitary NB eigenvalues.

\Cref{fig:examples}g is similar to \Cref{fig:examples}f though it shows that the case when there are multiple components for a single $q$ with different $r(S_q)$ is not unique to the existence of multi-edges.
Furthermore, the graph $X$ in \Cref{fig:examples}g also contains NB chains of length $5$ and $7$.
This means that $S_q$ is non-empty for $q \in \{ 5, 7, 10, 14 \}$.
However the expression $|E(S_q)| - |V(S_q)| + 1$ is always zero in these cases.

\subsection{Algorithm}

Note the formulas in \Cref{thm:main} do not involve a single matrix computation nor in fact computing the NB matrix itself.
\Cref{alg:main} computes these formulas efficiently.

\begin{algorithm}
\caption{Compute geometric multiplicity of unitary eigenvalue}
\label{alg:main}
\begin{algorithmic}[1]

\Function{ComputeGeometricMultiplicity}{$X, q$}

\State $\mathcal{C}_q \leftarrow$ \Call{GetChains}{$X, q$} \Comment Stage 1: get chains and initialization
\State $\overline{\mathcal{C}}_q \leftarrow$ graph whose edge set is $\mathcal{C}_q$
\State $\mathcal{S}_q \leftarrow$ \Call{ConnectedComponents}{$\overline{\mathcal{C}}_q$}
\State $m \leftarrow 0$
\If{$q$ is odd} \Comment Stage 2a: compute $m$ when $q$ is odd
\ForAll{$S_q \in \mathcal{S}_q$}
\State $m = m + |E(S_q)| - |V(S_q)| + 1$
\EndFor
\State \textbf{return} $m$
\Else \Comment Stage 2b: compute $m$ when $q$ is even
\ForAll{$S_q \in \mathcal{S}_q$}
\State $r(S_q) \leftarrow$ \Call{MaximumSubdivisionNumber}{$S_q$}
\If{$q$ divides $r(S_q)$}
\State $m = m + |E(S_q)| - |V(S_q)| + 1$
\Else \Comment $q/2$ divides $r(S_q)$
\State coll$(S_q, q) \leftarrow$ \Call{Collapse}{$S_q, q$}
\If{\Call{IsBipartite}{coll$(S_q, q)$}}
\State $m = m + |E(S_q)| - |V(S_q)| + 1$
\Else
\State $m = m + |E(S_q)| - |V(S_q)|$
\EndIf
\EndIf
\EndFor
\State \textbf{return} $m$
\EndIf
\EndFunction
\end{algorithmic}
\end{algorithm}

To discuss the time complexity of \Cref{alg:main}, let us first discuss the auxiliary functions used.
The function \textproc{ConnectedComponents}($X$) returns a collection of connected components of $X$, while the function \textproc{IsBipartite}($X$) returns $\textproc{True}$ if $X$ is bipartite and $\textproc{False}$ otherwise.
Both these algorithms are well known and can be computed in a single depth-first traversal of the graph.

The function \textproc{GetChains}($X, q$) returns a collection of NB chains of $X$ whose length is a multiple of $q$ or $q/2$.
This can be achieved in a single depth-first traversal of the graph, by keeping track of the current number of consecutive times that a node of degree $2$ was observed.
When a node of degree greater than $2$ is observed, the length of the current chain, if any, is checked for divisibility against $q$ and $q/2$.
If the check is successful, the current chain is stored, otherwise it is discarded.

The function \textproc{MaximumSubdivisionNumber}($X$) computes the maximum integer $r$ such that $X$ is the $r$-subdivision of some graph.
This is done by first obtaining the NB chains in $X$ while keeping track of the lengths of each chain found, and then computing the greatest common divisor of all lengths.

The function \textproc{Collapse}($X, q$) returns $\coll(X, q)$, the graph whose $q$-subdivision is $X$.
This is done in a single depth-first traversal of $X$, by adding a new edge to $\coll(X, q)$ for every $q$ edges of $X$ visited.

\begin{theorem}
Fix a graph $X$ with $m$ edges and $n$ nodes, and an integer $q>2$.
Let $k$ be the number of NB chains in $X$ whose lengths divide $q$ or $q/2$.
\Cref{alg:main} computes formulas \eqref{eqn:main-odd} and \eqref{eqn:main-even} in $O(m + k \log^2 m)$ time.
\end{theorem}

\begin{proof}
Let $p_1, \ldots, p_k$ be the lengths of the chains in $X$ whose lengths divide $q$ or $q/2$.
In the following we make the natural assumptions $q < m$ and $n < m$, and note $p_j < \sum_i p_i < m$ for each $j$.
Let us analyze \Cref{alg:main} line by line.

Line 2 requires a depth-first traversal, whose worst case scenario time complexity is $O(m + n)$.
Additionally, Line 2 performs two integer divisions per chain.
Using the Euclidean algorithm for integer division on a chain with length $p$, this is upper bounded by $O\left( (\log p + \log q)^2 \right)$; see \cite{knuth2014art,mathworld}.
Under our assumptions, this is $O(\log^2 m)$.
Thus the total complexity of Line 2 is upper bounded by $O(m + n + k \log^2 m)$.

Line 4 performs a depth-first traversal of $\overline{\mathcal{C}}_q$, which is always faster than a traversal of $X$ and is thus upper bounded by $O(m + n)$.

Line 6 takes constant time, as determining whether an integer is even can be done by looking at its last binary digit.

Lines 8, 14, 18, and 20 compute $|E(S_q)|$ and $|V(S_q)|$, both of which can be computed in a single graph traversal taking $O(|E(S_q)| + |V(S_q)|)$ time.
Since each of these lines is executed only once for each $S_q$, when the algorithm finishes running, the total time it takes to run these lines equals $O(|E(\overline{\mathcal{C}}_q)| + |V(\overline{\mathcal{C}}_q)|)$, upper bounded again by $O(m + n)$.

Line 12 computes the chains in $S_q$, taking $O(|E(S_q)| + |V(S_q)|)$ time, and then computes the greatest common divisor among the lengths of said chains.
As it must perform this once for each $S_q$, the total time taken is upper bounded by $O(m + n + \sum_i (\log p_i + \log q)^2) = O(m + n + k \log^2 m)$.

Line 13 takes $O((\log r(S_q) + \log q)^2)$.
Accounting for each component $S_q$, this yields $O(k \log^2 m)$.

Lines 16 and 17 take $O(|E(S_q)| + |V(S_q)|)$ time each, for a total of $O(m + n)$ when the algorithm finishes running.

All other lines take constant time.
The assumption $n < m$ yields the result.
\end{proof}

\begin{remark*}
Note that if the numbers $q, p_i, m$ fit in a single-precision variable, the time complexity for integer division becomes upper bounded by $O(\log m)$ instead of $O(\log^2 m)$.
Furthermore, when the number of nodes of degree $2$ in $X$ is small compared to $m$, or equivalently when $k \ll m$, the time complexity is essentially linear for large enough $m$.
\end{remark*}

\appendix
\numberwithin{equation}{section}
\section{Proof of equation (4.1)}
\label{app:indefinite}

Here we prove the forward direction of Equation \eqref{eqn:bbt} using the theory of indefinite linear algebra \cite{gohberg2005indefinite}. Throughout, let $G$ be a graph with NB matrix $\mathbf{B}$ with $\mathbf{B} \mathbf{v} = \lambda \mathbf{v}$ where $\lambda$ is nonreal.

\subsection*{The flip matrix and indefinite inner product}

The NB matrix $\mathbf{B}$ possesses a particular symmetry. There exists an operator $\mathbf{P}$ such that $\mathbf{B}\mathbf{P} = \mathbf{P}\mathbf{B}^*$.  We call $\mathbf{P}$ \emph{the flip operator}. The flip is defined as $(\mathbf{P}\mathbf{v})_{k \to l} = \mathbf{v}_{l \to k}$, for any function $\mathbf{v}$.  In other words, $\mathbf{P}$ acts by flipping the orientation of each directed edge. When the rows and columns of $\mathbf{B}$ are ordered in the usual way (first $m$ components corresponding to one orientation, second $m$ components to the opposite orientation), we can write
\begin{equation}
\mathbf{P} = \begin{pmatrix}
\mathbf{0} & \mathbf{I} \\
\mathbf{I} & \mathbf{0}
\end{pmatrix},
\end{equation}
where $\mathbf{0}$ and $\mathbf{I}$ are the $m \times m$ zero and the identity matrix, respectively.

Direct computation shows that $\mathbf{P}$ is involutory, invertible, Hermitian, and unitary, which implies that it defines a so-called indefinite inner product.  For two functions $\mathbf{x}, \mathbf{y} \in \mathbb{C}^{2m}$,
we have
\begin{equation}
[\mathbf{x}, \mathbf{y}]_\mathbf{P} := \mathbf{x}^{\top} \mathbf{P} \overline{\mathbf{y}}
= \sum_{i \to j} \mathbf{x}_{i \to j} \, \overline{\mathbf{y}_{j \to i}}.
\end{equation}

The following result is Theorem 4.2.4 of \cite{gohberg2005indefinite}.

\begin{theorem}[Gohberg-Lancaster-Rodman]
\label{thm:glr}
Let $\mathbf{A}$ be an $\mathbf{H}$-selfadjoint matrix (i.e., $\mathbf{A}^* =
\mathbf{H}^{-1}\mathbf{A}\mathbf{H}$ for some invertible Hermitian $\mathbf{H}$) and let
$\lambda, \mu$ be two distinct eigenvalues of $\mathbf{A}$. Then the root subspaces
corresponding to $\lambda$ and $\overline{\mu}$ are $\mathbf{H}$-orthogonal. In
particular, if $\lambda$ is nonreal, its root subspace is $\mathbf{H}$-neutral.
\end{theorem}

Since $\mathbf{B}\mathbf{P} = \mathbf{P}\mathbf{B}^*$, the matrix $\mathbf{B}$ is
$\mathbf{P}$-selfadjoint. Therefore, Theorem \ref{thm:glr} applies with $\mathbf{H} =
\mathbf{P}$.

\subsection*{The key lemma}

\begin{lemma}
\label{lem:key}
Let $\mathbf{B}\mathbf{v} = \lambda \mathbf{v}$ with nonreal $\lambda$. Then
\begin{equation}
\left( \left|\lambda\right|^2 - 1 \right) \mathbf{v}^* \mathbf{v}
= \sum_j \left|\vec{\mathbf{v}}^j\right|^2 \left( d_j - 2 \right).
\end{equation}
\end{lemma}

\begin{proof}
Since $\lambda$ is nonreal, Theorem \ref{thm:glr} implies that
$[\mathbf{v}, \mathbf{v}]_\mathbf{P} = 0$. We have
\begin{equation}
\label{eq:star}
0 = [\mathbf{v}, \mathbf{v}]_\mathbf{P}
= \mathbf{v}^\top \mathbf{P} \overline{\mathbf{v}}
= \sum_{i \to j} \mathbf{v}_{i \to j} \overline{\mathbf{v}}_{j \to i}
=: \star.
\end{equation}

The eigenvalue equation $\mathbf{B}\mathbf{v} = \lambda \mathbf{v}$, gives $\lambda \mathbf{v}_{k \to l} = \vec{\mathbf{v}}^k - \mathbf{v}_{l \to k}$.  We can now rewrite $\star$ as
\begin{align}
\star
= \sum_{i \to j} \overline{\mathbf{v}}_{j \to i}
   \left( \vec{\mathbf{v}}^j - \lambda \mathbf{v}_{j \to i} \right) 
= \sum_j \vec{\mathbf{v}}^j \sum_i a_{ij} \overline{\mathbf{v}}_{j \to i}
   - \lambda \sum_{i \to j} \overline{\mathbf{v}}_{j \to i} \mathbf{v}_{j \to i} 
= \sum_j \vec{\mathbf{v}}^j \overline{\vec{\mathbf{v}}_j}
   - \lambda \mathbf{v}^* \mathbf{v}.
\end{align}

Taking the complex conjugate and observing that $\mathbf{v}^* \mathbf{v} \in \mathbb{R}$,
we obtain
\begin{equation}
\label{eq:starstar}
0 = \sum_j \overline{\vec{\mathbf{v}}^j} \vec{\mathbf{v}}_j
  - \overline{\lambda} \mathbf{v}^* \mathbf{v}.
\end{equation}

By Lemma 4.2, we have $(d_j - 1) \vec{\mathbf{v}}^j = \lambda \vec{\mathbf{v}}_j$ for
each node $j$. Substituting into equation \eqref{eq:starstar} yields
\begin{equation}
0 = \sum_j \overline{\vec{\mathbf{v}}^j} \lambda^{-1} (d_j - 1) \vec{\mathbf{v}}^j
  - \overline{\lambda} \mathbf{v}^* \mathbf{v}
= \lambda^{-1} \sum_j \left|\vec{\mathbf{v}}^j\right|^2 (d_j - 1)
  - \overline{\lambda} \mathbf{v}^* \mathbf{v}.
\end{equation}

On the other hand, we can also rewrite equation \eqref{eq:star} as
\begin{align}
\star
&= \sum_{i \to j} \overline{\mathbf{v}}_{j \to i} \lambda^{-1}
   \left( \vec{\mathbf{v}}^i - \mathbf{v}_{j \to i} \right) \\
&= \lambda^{-1} \sum_i \vec{\mathbf{v}}^i \sum_j a_{ij} \overline{\mathbf{v}}_{j \to i}
   - \lambda^{-1} \sum_{i \to j} \overline{\mathbf{v}}_{j \to i} \mathbf{v}_{j \to i} \\
&= \lambda^{-1} \sum_i \vec{\mathbf{v}}^i \overline{\vec{\mathbf{v}}^i}
   - \lambda^{-1} \mathbf{v}^* \mathbf{v}.
\end{align}

Therefore,
\begin{equation}
\label{eq:sum-into}
\sum_i \left|\vec{\mathbf{v}}^i\right|^2 = \mathbf{v}^* \mathbf{v}.
\end{equation}

Multiplying by $\lambda$ and subtracting \eqref{eq:sum-into}, we obtain
\begin{equation}
0 = \sum_j \left|\vec{\mathbf{v}}^j\right|^2 (d_j - 1)
  - |\lambda|^2 \mathbf{v}^* \mathbf{v}
  - \sum_j \left|\vec{\mathbf{v}}^j\right|^2
  + \mathbf{v}^* \mathbf{v},
\end{equation}
which simplifies to the desired result.
\end{proof}

\subsection*{Proof of equation (4.1)}

\begin{proposition}
\label{prop:unitary-nonleaky}
Let $\mathbf{B}\mathbf{v} = \lambda \mathbf{v}$ with $|\lambda| = 1$. Then
$\mathbf{B}^*\mathbf{B}\mathbf{v} = \mathbf{B}\mathbf{B}^*\mathbf{v} = \mathbf{v}$.
\end{proposition}

\begin{proof}
By Lemma \ref{lem:key}, we have
\begin{equation}
0 = \left( |\lambda|^2 - 1 \right) \mathbf{v}^* \mathbf{v}
  = \sum_j \left|\vec{\mathbf{v}}^j\right|^2 (d_j - 2).
\end{equation}

Since each term in the sum is non-negative, we must have
$\left|\vec{\mathbf{v}}^j\right|^2 (d_j - 2) = 0$ for each node $j$. Therefore,
$(d_j - 2) \vec{\mathbf{v}}^j = 0$ for all $j$.

By Lemma 4.1, we have
\begin{equation}
(\mathbf{B}^*\mathbf{B}\mathbf{v})_{k \to l}
= (d_l - 2) \vec{\mathbf{v}}^l + \mathbf{v}_{k \to l}
= 0 + \mathbf{v}_{k \to l}
= \mathbf{v}_{k \to l}.
\end{equation}

Similarly, using Lemma 4.2, which shows that $(d_k - 2) \vec{\mathbf{v}}^k = 0$ implies
$(d_k - 2) \vec{\mathbf{v}}_k = 0$, and again by Lemma 4.1, we have
$\mathbf{B}\mathbf{B}^*\mathbf{v} = \mathbf{v}$.
\end{proof}


\section*{Acknowledgments}
This work started while the author was at Northeastern University’s Network Science Institute and supported in part by NSF IIS-1741197.
It continued while the author was at the Max Planck Institute for Mathematics in the Sciences.
The author thanks Gabor Lippner and Tina Eliassi-Rad for many invaluable conversations, and Blevmore Labs for consulting services regarding the construction of Figures \ref{fig:nbm-doodle} and \ref{fig:bbt}.

\addcontentsline{toc}{section}{References}
\bibliographystyle{plain}
\bibliography{references}

@article {northshield1998,
    AUTHOR = {Northshield, S.},
     TITLE = {A note on the zeta function of a graph},
   JOURNAL = {J. Combin. Theory Ser. B},
  FJOURNAL = {Journal of Combinatorial Theory. Series B},
    VOLUME = {74},
      YEAR = {1998},
    NUMBER = {2},
     PAGES = {408--410},
}

@article {Grindrod2018,
    AUTHOR = {Grindrod, P. and Higham, D. J. and Noferini, V.},
     TITLE = {The deformed graph {L}aplacian and its applications to network
              centrality analysis},
   JOURNAL = {SIAM J. Matrix Anal. Appl.},
  FJOURNAL = {SIAM Journal on Matrix Analysis and Applications},
    VOLUME = {39},
      YEAR = {2018},
    NUMBER = {1},
     PAGES = {310--341},
      ISSN = {0895-4798},
   MRCLASS = {05C50 (05C82 15A18 15A54)},
  MRNUMBER = {3769701},
MRREVIEWER = {Alexander Farrugia},
       DOI = {10.1137/17M1112297},
       URL = {https://doi-org.ezproxy.neu.edu/10.1137/17M1112297},
}

@techreport{petersen2012matrix,
    author       = "Brandt, L. and Petersen, K. and Pedersen, M. S.",
    title        = "The Matrix Cookbook",
    year         = "2012",
    month        = "November",
    institution    = "Technical University of Denmark",
    note         = "Version 20121115, \url{http://www2.compute.dtu.dk/pubdb/pubs/3274-full.html}",
    url          = "http://www2.compute.dtu.dk/pubdb/pubs/3274-full.html",
}

@book{horn2012matrix,
  title={Matrix analysis},
  author={Horn, R. A. and Johnson, C. R.},
  year={2012},
  publisher={Cambridge University Press}
}

@article{merris1994laplacian,
  title={Laplacian matrices of graphs: A survey},
  author={Merris, R.},
  journal={Linear algebra and its applications},
  volume={197},
  pages={143--176},
  year={1994},
  publisher={Elsevier}
}

@article{cvetkovic2007signless,
  title={Signless Laplacians of finite graphs},
  author={Cvetkovi{\'c}, D. and Rowlinson, P. and Simi{\'c}, S. K.},
  journal={Linear Algebra and its Applications},
  volume={423},
  number={1},
  pages={155--171},
  year={2007},
  publisher={Elsevier}
}

@book{rowlinson2004, 
    series={London Mathematical Society Lecture Note Series}, 
    title={Spectral generalizations of line graphs: On graphs with least eigenvalue -2}, 
    publisher={Cambridge University Press}, 
    author={Cvetkovi\'{c}, D. and Rowlinson, P. and Simic, S.}, 
    year={2004}, 
    collection={London Mathematical Society Lecture Note Series},
}

@article{simic2015,
  title = {Graphs with least eigenvalue -2: Ten years on},
  author = {Cvetkovi\'{c}, D. and Rowlinson, P. and Simi\'{c}, S.},
  journal = {Linear Algebra and its Applications},
  volume = {484},
  pages = {504 - 539},
  year = {2015},
  issn = {0024-3795}
}

@book{godsil2013algebraic,
  title={Algebraic graph theory},
  author={Godsil, C. and Royle, G. F.},
  year={2013},
  publisher={Springer Science \& Business Media}
}

@article{kotani,
    AUTHOR = {Kotani, M. and Sunada, T.},
     TITLE = {Zeta functions of finite graphs},
   JOURNAL = {J. Math. Sci. Univ. Tokyo},
  FJOURNAL = {The University of Tokyo. Journal of Mathematical Sciences},
    VOLUME = {7},
      YEAR = {2000},
    NUMBER = {1},
     PAGES = {7--25},
      ISSN = {1340-5705},
   MRCLASS = {68R10 (05C50)},
  MRNUMBER = {1749978},
MRREVIEWER = {Christian Radoux},
}

@article{angel2007non,
    AUTHOR = {Angel, O. and Friedman, J. and Hoory, S.},
     TITLE = {The non-backtracking spectrum of the universal cover of a graph},
   JOURNAL = {Trans. Amer. Math. Soc.},
  FJOURNAL = {Transactions of the American Mathematical Society},
    VOLUME = {367},
      YEAR = {2015},
    NUMBER = {6},
     PAGES = {4287--4318},
      ISSN = {0002-9947},
   MRCLASS = {05C50 (05C38)},
  MRNUMBER = {3324928},
MRREVIEWER = {Michael A. Tait}
}

@article{martin2014localization,
  title = {Localization and centrality in networks},
  author = {Martin, T. and Zhang, X. and Newman, M. E. J.},
  journal = {Phys. Rev. E},
  volume = {90},
  issue = {5},
  number = {052808},
  year = {2014},
  publisher = {American Physical Society}
}

@article{torres2020non,
  title={Non-backtracking Spectrum: Unitary Eigenvalues and Diagonalizability},
  author={Torres, L.},
  journal={Preprint. arXiv:2007.13611},
  year={2020}
}

@article{lin2019non,
  title={Non-backtracking centrality based random walk on networks},
  author={Lin, Y. and Zhang, Zh.},
  journal={The Computer Journal},
  volume={62},
  number={1},
  pages={63--80},
  year={2019},
  publisher={Oxford University Press}
}

@phdthesis{horton2006ihara,
    title    = {Ihara zeta functions of irregular graphs},
    school   = {University of California, San Diego},
    author   = {Horton, M. D.},
    year     = {2006},
}

@incollection{hashimoto1989zeta,
  title={Zeta functions of finite graphs and representations of p-adic groups},
  author={Hashimoto, K.-{I}.},
  booktitle={Automorphic forms and geometry of arithmetic varieties},
  pages={211--280},
  year={1989},
  publisher={Elsevier}
}

@article {cooper2009properties,
    AUTHOR = {Cooper, Y.},
     TITLE = {Properties determined by the {I}hara zeta function of a graph},
   JOURNAL = {Electron. J. Combin.},
  FJOURNAL = {Electronic Journal of Combinatorics},
    VOLUME = {16},
      YEAR = {2009},
    NUMBER = {1}
}

@article{bordenave2018nonbacktracking,
  title={{Nonbacktracking spectrum of random graphs: Community detection and nonregular Ramanujan graphs}},
  author={Bordenave, C. and Lelarge, M. and Massouli{\'e}, L.},
  journal={Annals of probability: An official journal of the Institute of Mathematical Statistics},
  volume={46},
  number={1},
  pages={1--71},
  year={2018},
  publisher={Institute of Mathematical Statistics}
}

@book{terras2010zeta,
  title={Zeta functions of graphs: {A} stroll through the garden},
  author={Terras, A.},
  series={{C}ambridge studies in advanced mathematics},
  volume={128},
  year={2010},
  publisher={Cambridge University Press}
}

@article{bass1992ihara,
  title={{The Ihara-Selberg zeta function of a tree lattice}},
  author={Bass, H.},
  journal={International Journal of Mathematics},
  volume={3},
  number={06},
  pages={717--797},
  year={1992},
  publisher={World Scientific}
}

@article{torres2021nonbacktracking,
  title={Nonbacktracking eigenvalues under node removal: X-centrality and targeted immunization},
  author={Torres, L. and Chan, K. S. and Tong, H. and Eliassi-Rad, T.},
  journal={SIAM Journal on Mathematics of Data Science},
  volume={3},
  number={2},
  pages={656--675},
  year={2021},
  publisher={SIAM}
}

@article{Torres2019,
  author    = {Torres, L. and
            Su{\'{a}}rez{-}Serrato, P. and
             Eliassi{-}Rad, T.},
  title     = {Non-backtracking cycles: length spectrum theory and graph mining applications},
  journal   = {Applied Network Science},
  volume    = {4},
  number    = {1},
  year      = {2019}
}

@misc{mathworld,
    author   = {Weisstein, E. W.},
    title    = {{E}uclidean {A}lgorithm. {From MathWorld---A Wolfram Web Resource}},
    url      = {\url{https://mathworld.wolfram.com/EuclideanAlgorithm.html}},
    note     = {Last visited on 03/05/2022},
}

@book{knuth2014art,
  title={Art of computer programming, volume 2: Seminumerical algorithms},
  author={Knuth, D. E.},
  year={2014},
  publisher={Addison-Wesley Professional}
}

@article{tao2017random,
  title={Random matrices have simple spectrum},
  author={Tao, T. and Vu, V.},
  journal={Combinatorica},
  volume={37},
  number={3},
  pages={539--553},
  year={2017},
  publisher={Springer}
}

@book{gohberg2005indefinite,
    title     = {Indefinite Linear Algebra and Applications},
    author    = {Gohberg, Israel and Lancaster, Peter and Rodman, Leiba},
    year      = {2005},
    publisher = {Birkh\"auser Basel},
    doi       = {10.1007/b137517}
  }
\end{document}